\documentclass{cmslatex}
\usepackage{etex}
\usepackage[paperwidth=7in, paperheight=10in, margin=.875in]{geometry}
 \usepackage[colorlinks,linkcolor=red,anchorcolor=green,citecolor=blue]{hyperref}
\usepackage{etex}
\usepackage{amsfonts,amssymb}
\usepackage{amsmath}
\usepackage{dashbox}
\usepackage{cite}
\usepackage{paralist}
\usepackage{array}
\usepackage{ragged2e}
\usepackage[english]{babel} 
\usepackage{ifthen}
\usepackage{framed, color}
\usepackage{bbm}		
\usepackage{graphicx}
\usepackage{epstopdf}
\usepackage{lmodern}
\usepackage{algorithmic}
\usepackage{float}
\usepackage{wrapfig} 
\usepackage{graphicx}
\usepackage{cancel}
\usepackage{mathrsfs}
\usepackage{upgreek}
\usepackage{tikz}	
\usetikzlibrary{arrows,shapes,decorations.pathreplacing,decorations.markings,patterns,calligraphy}
\usepackage{mathtools} 
\usepackage{ragged2e}
\usepackage{tabularx}
\usepackage{nicefrac}
\usepackage{multirow}
\usepackage{enumerate}
\usepackage{yhmath}

\newcommand*\xbar[1]{%
	\hbox{%
		\vbox{%
			\hrule height 0.5pt 
			\kern0.5ex
			\hbox{%
				\kern-0.1em
				\ensuremath{#1}%
				\kern-0.1em
			}%
		}%
	}%
} 

\newcommand*{\colorboxed}{}
\def\colorboxed#1#{%
	\colorboxedAux{#1}%
}
\newcommand*{\colorboxedAux}[3]{%
	\begingroup
	\colorlet{cb@saved}{.}%
	\color#1{#2}%
	\boxed{%
		\color{cb@saved}%
		#3%
	}%
	\endgroup
}

\DeclareFontFamily{U}{mathx}{\hyphenchar\font45}
\DeclareFontShape{U}{mathx}{m}{n}{<-> mathx10}{}
\DeclareSymbolFont{mathx}{U}{mathx}{m}{n}
\DeclareMathAccent{\widebar}{0}{mathx}{"73}

\def \vMAX{\textbf{v}_{\max}}

\def \indikator{\mathbbm{1}}
\def \zeros{\mathbb{O}}

\def \d{\, \textup{d}}
\def \T{\textup{T}}

\def \diag{\textup{diag}}

\def \A{\widehat{\,\mathcal{A}\; }}

\def \JoneDim{\widehat{\;\textup{\textbf{J}}}}
\def \JtildeoneDim{\widetilde{\;\textup{\textbf{J}}}}
\def \J{\widehat{\;\textup{\textbf{J}}}_{\vec{n}} }

\def \Ji{\widehat{\;\textup{\textbf{J}}}_{e_i} }

\def \JC{\widetilde{\;\textup{\textbf{J}}}_{\vec{n}} }

\def \x{\textup{x}}
\def \D{\textup{D}}

\def \GK{\Pi_K}

\def \u{\textup{\textbf{u}}}

\def \wtilde{\widetilde{\textit{\textbf{w}}}}

\def \NumFlux{\widehat{\mathcal{F}}}
\def \NumFluxC{\widetilde{\mathcal{F}}}
\def \NumFluxX{\widetilde{\mathcal{F}}^{(\v)}}

\def \uI{\widehat{\,\textup{\textbf{u}}_1}}
\def \uII{\widehat{\,\textup{\textbf{u}}_2}}
\def \uIII{\widehat{\,\textup{\textbf{u}}_i}}

\def \domain{\mathbb{D}}

\def \ubar{\xbar{\textup{\textbf{u}}}}

\def \DxII{ \Delta y }
\def \DxI{ \Delta x }

\def \ubar{\xbar{\textup{\textbf{u}}}}
\def \ubarL{\xbar{\textup{\textbf{u}}}_{\;\ell}}
\def \ubarR{\xbar{\textup{\textbf{u}}}_{\;r}}
\def \ubarK{\xbar{\textup{\textbf{u}}}_{\;q}}

\def \wbar{\xbar{\textup{\textbf{w}}}}
\def \wbarL{\xbar{\textup{\textbf{w}}}_{\;\ell}}
\def \wbarR{\xbar{\textup{\textbf{w}}}_{\;r}}
\def \wbarK{\xbar{\textup{\textbf{w}}}_{\;q}}

\def \wbar{\xbar{\textup{\textbf{w}}}}

\def \w{\widehat{\textup{\textbf{w}}}}

\def \P{\mathcal{P}}
\def \DP{\mathcal{D}}
\def \VP{\mathcal{V}}

\def \sign{\textup{sign}}
\def \ui{\widehat{\,\u_i}}

\def \v{\widehat{\textup{\textbf{v}}}}
\def \H{\widehat{\textup{\textbf{H}}}}

\def \Legendre{\phi}

\def \Feffective{\widetilde{F}^{(\textup{eff})}}
\def \UNIT{\widehat{\;e_1}}

\def \treshold{\textup{T}}

\DeclareMathOperator*{\argmin}{argmin}
\definecolor{gruen}{rgb}{0,0.5,0}
\definecolor{myred}{rgb}{1,0.3,0.1} 
\definecolor{myblue}{rgb}{0.196,0.196,0.694}

\renewcommand{\theequation}{\arabic{section}.\arabic{equation}}

   \allowdisplaybreaks
\begin{document}

 \title{Description of random level sets \\by polynomial chaos expansions}


\author{Markus Bambach\thanks{ETH Z\"urich, 8005 Z\"urich, Switzerland, mbambach@ethz.ch}
	\and Stephan Gerster\thanks{University of Mainz,  55122 Mainz, Germany, stephan.gerster@gmail.com} 
	\and  Michael Herty\thanks{RWTH Aachen University, 52062 Aachen, Germany, herty@igpm.rwth-aachen.de}
	\and {Aleksey~Sikstel}\thanks{TU Darmstadt, 64289 Darmstadt, Germany, sikstel@mathematik.tu-darmstadt.de}}


\pagestyle{myheadings} \markboth{Description of random level sets by polynomial chaos expansions}{Markus Bambach, Stephan Gerster, Michael Herty, Aleksey Sikstel} \maketitle

\begin{abstract}
We present a novel approach to determine the evolution of level sets under uncertainties in their velocity fields. This leads to a stochastic description of  level sets.  
To compute the quantiles of random level sets, we use the stochastic Galerkin method for a hyperbolic reformulation of the equations for the propagation of level sets. A novel intrusive Galerkin formulation is presented and proven to be hyperbolic. It induces a corresponding finite-volume scheme that is specifically taylored to uncertain  velocities.
\end{abstract}

\begin{keywords}
Level sets, uncertainty quantification, Hamilton-Jacobi equations, hyperbolic conservation laws, stochastic Galerkin,  finite-volume method
\end{keywords}

 \begin{AMS}
	35F21; 37L45; 60D05; 60H15
\end{AMS}

\section{Introduction}
The tracking and representation of moving interfaces is of interest in numerous applications ranging from material science~\cite{Imran2021}, chemical simulations~\cite{Sethian1999} to  fluid-dynamics~\cite{Sethian2003}.  
Among others, level set methods~\cite{Dervieux1980,Osher1988,Sethian1999} are used to tackle these problems. 
The main idea is to describe a moving interface as the zero-level set of a so-called level set function. The moving boundary of this level set is 
then described by a partial differential equation (PDE).  
The PDE that describes the level sets is a Hamilton-Jacobi equation. Those are equivalent to a hyperbolic form in the sense of viscosity solutions~\cite{Crandall1983,Subbotina2017,Caselles1992}. 
The hyperbolic form is often preferred for numerical purposes. However, also parabolic approximations are considered~\cite{Beckermann2007} that model more diffusive interfaces.

Although level set equations have been studied intensively since the fundamental works~\cite{Dervieux1980,Osher1988} in the 1980s, there are several open problems in the context of 
controllability, robustness and uncertainty quantifications, which 
remain an active field of research~\cite{Preusser,Nouy2009,Pettersson2019}. 

This work contributes to questions of uncertainty and robustness. 
For instance, 
model parameters may be uncertain
due to noisy measurements. Mathematical models do not exactly describe the true physics due to epistemic uncertainties, e.g.~in constitutive equations for material models. In particular, if the uncertainty affects the level set, a challenge occurs, since the  use of stochastic level sets  leads to stochastic domain boundaries. Hence, the zero-level set is not a single closed curve anymore. Instead, there may be a band of possibly arbitrary thickness which contains all possible locations of the random zero-level set~\cite{Preusser}. 
Therefore,  meaningful statistics are of interest as e.g.~confidence bounds, mean, variance and higher moments of the location of the boundary. To compute these statistics, the whole probability distribution must be available.

Typically, 
Monte-Carlo and stochastic collocation methods~\cite{S4} are used to quantify uncertainties. Then, a deterministic problem is solved for each realization. However, in each time step 
only the solution corresponding to a particular sample
or quadrature point is available. 
Only after all simulations are completed, the statistics can be determined. 

In many engineering applications, for instance forming processes~\cite{Imran2021}, the statistics of interest must be computed online. To this end, we follow an intrusive stochastic Galerkin approach. 
The functional dependence of the level set function on the stochastic input is described a priori by a series expansion and a 
Galerkin projection is used to obtain deterministic evolution equations for the coefficients of this series. This approach has been applied to the \emph{parabolic} approximation~\cite{Beckermann2007,Preusser} of random level set equations. 

The aim of this work is to introduce a theoretical framework for the  \emph{hyperbolic} formulation. 
It is well-known that the stochastic Galerkin method does not necessarily transfer hyperbolicity to the Galerkin formulation. 
Still, there are  successful applications to many scalar conservation laws,  when the resulting Jacobians of the flux are symmetric. Then, well-balanced schemes have been developed~\cite{H5} and  a maximum-principle has been ensured~\cite{kusch2019maximum}. 
However, the level set equations with uncertain velocity may result --- even in the one-dimensional scalar case --- in a non-symmetric Jacobian of the stochastic Galerkin system. Furthermore, a Hamilton-Jacobi form in \emph{multiple} spatial dimenions, results in a \emph{system} of hyperbolic equations. 

There are many examples that show  loss of hyperbolicity, when the stochastic Galerkin approach is applied directly to hyperbolic systems. 
To this end, auxiliary variables have been introduced to established wellposedness results. For instance, entropy-entropy flux pairs can be obtained by an expansion in entropy variables, i.e.~the gradient of the deterministic entropy~\cite{H0,H1}.  
Roe variables, which include the square root of the density, preserve hyperbolicity for Euler equations~\cite{Roe1,S5,GersterJCP2019}. 
One drawback of introducing auxiliary variables is an additional computational overhead that arises from an optimization problem, which is required to calculate the auxiliary variables. Recently, a hyperbolic stochastic Galerkin for  the shallow water equations has been presented that  neither requires  auxiliary variables nor any transform, since the  
Jacobian is shown to be similar to a symmetric 
matrix~\cite{Epshteyn2021,Epshteyn2022}. 


To establish hyperbolicity for the level set equations in multiple dimensions with uncertainties in velocities, we follow a strategy that is based both on the introduction of auxiliary variables and on the symmetrization of the Jacobian~\cite{Epshteyn2021,Epshteyn2022}. 
First, an expression for the Euclidean norm is presented. It is \emph{fully intrusive} by which we mean that all integrals are \emph{exactly} computed in a precomputation step and not during a simulation. 
The intrusive expression is inspired by the concept of Roe variables in~\cite{S5}, namely to use the square root as an auxiliary variable, which is obtained by minimizing a convex function~\cite{GersterHertyCicip2020}. 
Our main theorem  
states matrix similarities for a conservative and a capacity form. Both formulations are hyperbolic if the randomness arises only from initial values. 
In the context of level set equations, however, uncertain velocities are of interest that may lead to complex characteristic speeds in a conservative formulation. This issue is circumvented by the capacity form  that ensures hyperbolicity even for random velocities. 
Those pose also serious challenges in the computational approximation. 
The numerical discretization, restricted by the CFL-condition, must account for all appearing wave speeds. 
However, applying the stochastic Galerkin method to random velocities, results in additional waves that are taken into account by  \emph{non-uniform effective grids}.
Here, we will follow the concept of a \emph{capacity-form differencing}~scheme~\cite[Sec.~6.16, Sec.~6.17]{Leveque}. 
The idea is to introduce an \emph{effective non-uniform} space discretization that is related to a \emph{uniform computational} domain by a coordinate transform, which is indirectly given by the waves that result from random velocities.

The remainder of this section considers a motivating example to material deforming. Formulations in Hamilton-Jacobi and hyperbolic form are presented and the  new concept of random level sets is proposed for  two-dimensional level set equations. 

\subsection{Grain size evolution}
Let $\domain\subseteq\mathbb{R}^2$ be a domain (the workpiece) that is deformed by deformation processing at elevated temperatures. The deformation process yields the evolution of a microstructural feature $g$ (grain size) in the material, which is a scalar value attached to each material point. During the deformation process, the grain size $g$ evolves in such a way that the domain $\domain = \domain_0\cup \domain_1$ can be split into subregions~$
 \domain_0\coloneqq \big\{
\x \in\mathbb{R}^2
\; \big| \; g(\x)<g^* 
 \big\} 
 $ 
 and~$
 \domain_1\coloneqq \big\{
 \x \in\mathbb{R}^2
 \; \big| \; g(\x)\geq g^* 
 \big\} 
 $. 
In a two-dimensional model the \textbf{zero-level set} 
$$
\Gamma(t)\coloneqq
\Big\{ 
\x \in\mathbb{R}^2
\ \Big| \
\varphi(t,\x) = 0
\Big\}
=
\partial \domain_0 \cap \partial \domain_1
$$
is implicitly described by a \textbf{level set function}~$\varphi \; :\; \mathbb{R}^2\rightarrow \mathbb{R}$. For instance, when $g(\x)$ describes the grain size, it is given by the partial differential equation
$$
\partial_t \varphi(t,\x) + \nabla_{\x} g(\x) \cdot \nabla_{\x} \varphi(t,\x) = 0
\quad\text{with initial data}\quad
 \varphi(0,x) = \varphi_0(x).$$
Typically, 
only displacements along the normal direction of the level set are of interest. Hence, we replace the Jacobian of the grain size by
\begin{equation}\label{displacement}
\nabla_{\x} g(\x)=v(\x) \frac{ \nabla_{\x} \varphi(t,\x)}{\big\lVert\nabla_{\x} \varphi(t,\x) \big\rVert},
\end{equation}
where~$\lVert\cdot\rVert$ denotes the Euclidean norm  and~$v(\x)\in \mathbb{R}$ the velocity. 
We obtain the \textbf{level set equations} as a Hamilton-Jacobi equation
\begin{equation}\label{detLevelSet}
\partial_t \varphi(t,\x) + v(\x) \big\lVert \nabla_{\x} \varphi(t,\x) \big\rVert   = 0.
\end{equation}


\noindent
The previous application motivates the more general discussion.  
Following the notation in~\cite{HamiltonJacobiJin},  we consider the random
\begin{alignat*}{8}
&\textbf{Hamilton-Jacobi equations}\quad
&&\partial_t \varphi(t,\x)
&&+
H\Big( \nabla_{\x} \varphi(t,\x),\x \Big)&&=0, \\
&\textbf{hyperbolic conservation laws}\quad
&&\partial_t \u(t,\x)
&&+
\nabla_{\x}
H\Big(  \u(t,\x),\x\Big)&&=0,
\quad
\u(t,\x) \coloneqq \nabla_{\x}  \varphi(t,\x)
\end{alignat*}
for one~$x\in\mathbb{R}$ and two dimensions~$\x=(x_1,x_2)\in\mathbb{R}^2$.
In the deterministic case, solutions to those are equivalent in the sense of viscosity solutions~\cite{Caselles1992,Jin1998}. Clearly, with the choice of the Hamiltonian
\begin{equation}\label{RandomHamiltonian}
H(\u,\x)
=
 v(\x) \lVert  \u \rVert 
\end{equation}
the level set equation~\eqref{detLevelSet} is recovered. 

\bigskip
\begin{remark}
In fact, the hyperbolic formulation can be interpreted also as a \emph{system} of \emph{hyperbolic balance laws} that read in the deterministic case as
\begin{equation}\label{Strongly}
\partial_t 
\begin{pmatrix}
\u(t,\x) \\ \varphi(t,\x) 
\end{pmatrix}
+
\nabla_{\x}
\begin{pmatrix}
H\big(  \u(t,\x),\x \big) \\
0
\end{pmatrix}
=-
\begin{pmatrix}
0 \\
H\big(  \u(t,\x),\x \big) 
\end{pmatrix}.
\end{equation}
As remarked in~\cite{HamiltonJacobiJin}, the balance law~\eqref{Strongly} is a \textbf{strongly hyperbolic system} in the sense that characteristic speeds are real and the
Jacobian of the flux function admits a complete set of eigenvectors~\cite{Godlewski1998}. The quasilinear form  reads as
\begin{equation}\label{quasilinear}
\begin{aligned}
&\partial_t 
\u(t,\x) 
+
\sum\limits_{i=1}^2
H_i'\big(  \u(t,\x),\x \big) 
\nabla_{\x} \u_i(t,\x)
=-
\nabla_{\x}
H(  \u,\x )\big|_{\u=\u(t,\x)} \\
&\text{for}\quad
H_i'\big(  \u(t,\x),\x \big)  
\coloneqq
\partial_{\u_i}
H\big(  \u,\x \big)\big|_{\u=\u(t,\x)}  .
\end{aligned}
\end{equation}

\end{remark}


\noindent
For this particular example, the level set equations in two dimensions, i.e.~$\x=(x_1,x_2)^\T $ and $\u(t,\x)\in\mathbb{R}^2$, read in hyperbolic form  as
\begin{equation}\label{2dLevelSet}
	\begin{aligned}
		&\partial_t \u(t,\x)
		+
		\partial_{x_1}
		f_1\big(\u(t,\x),v(\x)\big)
		+
		\partial_{x_2}
		f_2\big(\u(t,\x),v(\x)\big)
		=0 \\
		& 
		\text{with flux functions}\quad
		f_1(\u,v)
		=
		\begin{pmatrix}
			v \lVert \u \rVert \\ 0
		\end{pmatrix}
		\quad\text{and}\quad
		f_2(\u,v)
		=
		\begin{pmatrix}
			0 \\
			v \lVert \u \rVert
		\end{pmatrix}.
	\end{aligned}
\end{equation}

\noindent
The two-dimensional system~\eqref{2dLevelSet} is hyperbolic in the sense that for all unit vectors~${\vec{n}=(n_1,n_2)^\T}$ the matrix
\begin{equation}\label{DeterministicJacobian}
	n_1 \D_{\u} f_1(\u,v) 
	+ 
	n_2 \D_{\u} f_2(\u,v)
	=
	\frac{v}{\lVert \u \rVert}
	\begin{pmatrix}
		n_1 \u_1 & n_1 \u_2 \\
		n_2 \u_1 & n_2 \u_2
	\end{pmatrix}
\end{equation}
is diagonalizable with real eigenvalues and a complete set of eigenvectors. 
Note that the Jacobian~\eqref{DeterministicJacobian} reduces in the spatially one-dimensional case to~$
\D_{\u} f(\u,v) = v\, \partial_{\u} |\u|
$. 
Here, the derivative of the norm is understood as generalized gradient~\cite{Clarke1990,Frankowska1989, LipschitzFlux2001}. 

\subsection{Description of random level sets}
The ansatz~\eqref{displacement} heavily depends on an appropriate choice of the displacements~$v(\x,\xi)$ that are subject to uncertainties. Those are summarized in a random variable $\xi \, :\, \Omega\rightarrow \mathbb{R}$ that is defined on a probability space~$\big(\Omega,\mathcal{F}(\Omega),\mathbb{P}\big)$. 
Then, the displacement is for each fixed point in space also a random variable with realizations~$v\big(\x,\xi(\omega)\big)\in \mathbb{R}$ for $\omega\in\Omega$.
Since the level set function~$\varphi(t,\x,\xi)$ is random as well, the deterministic zero-level set must be extended to the stochastic case. We propose to consider the \textbf{quantile of the perturbed level set} 
\begin{equation}\label{randomLS}
\widehat{\Gamma_{\varepsilon,p}}(t)
\coloneqq
\Big\{ 
\x \in\mathbb{R}^2
\ \Big| \
\mathbb{P}\Big[
\big|\varphi(t,\x,\xi)\big| \leq \varepsilon
\Big] \geq p
\Big\}.
\end{equation}
To compute the probability in the set~\eqref{randomLS} the whole probability distribution of the level set function must be available at each time step, which motivates the stochastic Galerkin method. 
We start in Section~\ref{SectionGPC} from the hyperbolic form. Two stochastic Galerkin formulations are derived and analyzed in terms of hyperbolicity, namely in a \emph{conservative} and \emph{capacity} form. 
The main theorem ensures that the capacity form is always hyperbolic. 
Section~\ref{SectionFV} presents a hyperbolicity preserving finite-volume discretization that converges to the derived capacity form.

\section{Intrusive formulation}\label{SectionGPC}
To account for the random and space-dependent Hamiltonian~$
H(\u,\x,\xi)
=
v(\x,\xi) \lVert  \u \rVert 
$, 
the deterministic formulations are equipped with a random variable~$\xi$, i.e.
$$
\partial_t \varphi(t,\x,\xi)
+
H\Big( \nabla_{\x} \varphi(t,\x),\x,\xi \Big)=0
\quad\text{and}\quad
\partial_t \u(t,\x,\xi)
+
\nabla_{\x}
H\Big(  \u(t,\x),\x,\xi\Big)=0.
$$
The dependency of the solutions $\u(t,x,\cdot)$ and $\varphi(t,x,\cdot)$ on the stochastic input $\xi$ is described a priori in terms of orthogonal functions. 
For instance, 
\textbf{normalized Legendre polynomials} with uniform distribution~$\xi\in\mathcal{U}(0,1)$  are 
 recursively defined by
$$
\Legendre_0(\xi) = 1,\quad
\Legendre_1(\xi) =  \sqrt{3}  \xi,\quad
\Legendre_{k+1}(\xi) 
= 
\frac{\sqrt{2k+3}}{k+1}
\bigg(
\sqrt{2k+1}
\xi \Legendre_{k}(\xi) - \frac{k}{\sqrt{2k-1}} \Legendre_{k-1}(\xi) 
\bigg).
$$ 
\textbf{Hermite polynomials}  are given  by
$$
\widetilde{\phi}_0
=
1, 
\quad
\widetilde{\phi}_1
=
\xi, 
\quad
\widetilde{\phi}_{k+1}
=
\xi \widetilde{\phi}_k(\xi)
-
k\widetilde{\phi}_{k-1}
\quad
\text{with normalization}
\quad
\phi_k(\xi) \coloneqq
(k! )^{-\nicefrac{1}{2}}
\widetilde{\phi}_k(\xi)
$$
and 
are orthogonal to the normal distribution~$\xi\sim \mathcal{N}(0,1)$. 
More precisely, we have
$$
\mathbb{E}\Big[
\phi_i(\xi) \phi_j(\xi)
\Big]
=
\int \phi_i(\xi) \phi_j(\xi) \d \mathbb{P}
\eqqcolon
\langle \phi_{i}, \phi_{j}  \rangle_{\mathbb{P}}
=
\delta_{i,j}. 
$$
Then, a random state $\u(\xi)$ with gPC modes $\widehat{\u}\in\mathbb{R}^{K+1}$ is approximated by
$$
\Pi_K\big[\widehat{\u}\big](\xi)
\coloneqq
\sum\limits_{k=0}^K
\widehat{\u}_k \phi_k(\xi)
\quad\text{satisfying}\quad
\Big\lVert 
\Pi_K\big[\widehat{\u}\big](\xi) 
-
\u(\xi)
\Big\rVert_{\mathbb{P}}
\rightarrow 0
\quad\text{for}\quad
K\rightarrow \infty. 
$$
Similarly to~\cite{S15,S4,S18}, we express products and the second moment by
\begin{equation} \label{2ndMoment}
\begin{aligned}
&\widehat{\u}
\ast
\widehat{\textbf{q}}
\coloneqq
\P(\widehat{\u})
\widehat{\textbf{q}}
\qquad
\text{and}
\qquad
\widehat{\u}^{\ast 2}
\coloneqq 
\mathcal{R}(\widehat{\u})
\coloneqq
\P(\widehat{\u}) \widehat{\u} \\
&\text{for}\quad
\P(\widehat{\u})\coloneqq\sum\limits_{k=0}^K \widehat{\u}_k\mathcal{M}_k,
\quad
\mathcal{M}_k\coloneqq
\Big(
\langle \phi_k,\phi_j\phi_i \rangle_{\mathbb{P}}\Big)_{i,j=0,\ldots,K}.
\end{aligned}
\end{equation}
More precisely, the stochastic Galerkin matrix~\eqref{2ndMoment} defines a linear operator ${\P \ : \
\mathbb{R}^{|\widehat{\u}|} \mapsto 
\mathbb{R}^{|\widehat{\u}| \times |\widehat{\u}|}, \
\widehat{\u}
\mapsto
\P(\widehat{\u}) 
}$, where~$|\widehat{\u}|$ denotes the number of gPC modes. Hence, we use the notation~$\P(\widehat{\u}) \in\mathbb{R}^{d(K+1) \times d(K+1)  } $  for both the one-dimensional case with gPC modes~$\widehat{\u}= \uI\in\mathbb{R}^{K+1} $ and for two dimensions with~$\widehat{\u}= (\uI,\uII)^\T \in\mathbb{R}^{2(K+1)} $.  
According to~\cite[Lem.~3.1]{GersterHertyCicip2020},  
the random Euclidean norm~$\big\lVert \u(\xi) \big\rVert$  can be approximately represented in terms of the vector
\begin{align}
&\widehat{ \,
\lVert \u \rVert
\, }
\coloneqq
\mathcal{R}^{-1}\left(
\sum\limits_{i=1}^d
\uIII^{\ast 2}	
	\right)
=
\argmin
\Bigg\{
\frac{\hat{\alpha}^\T \P(\hat{\alpha}) \hat{\alpha} }{3}
\ - \
\hat{\alpha}^\T \,
\sum\limits_{i=1}^d
\uIII^{\ast 2}
\Bigg\} \label{ARGMIN}\\
&\text{satisfying}\quad
\widehat{ \,
	\lVert \u \rVert
	\, }
\ast
\widehat{ \,
	\lVert \u \rVert
	\, }
=
\mathcal{R}\Big( \widehat{ \,
	\lVert \u \rVert
	\, } \Big)
\quad
\text{and}
\quad
\mathbb{E}\bigg[\Big(
\Pi_K\left[
\widehat{ \,
	\lVert \u \rVert
	\, }
\right](\xi)
-
\big\lVert \u(\xi) \big\rVert
\Big)^2\bigg]
\overset{K\rightarrow \infty }{\longrightarrow }0,
\end{align}
where $\mathcal{R}^{-1} \,:\, \mathbb{R}^{K+1} \rightarrow \mathbb{R}^{K+1}$ is the inverse mapping of the second moment~\eqref{2ndMoment}. 
More precisely, see~\cite[Lem.~3.1]{GersterHertyCicip2020},  the mapping~$\mathcal{R}$ is bijective provided that the matrix~
\begin{equation}\label{AI}
\P\Big(
\widehat{ \,
	\lVert \u \rVert
	\, }
\Big)
\quad
\text{is strictly positive definite.}
\end{equation}

\noindent
Figure~\ref{ARGMIN} illustrates the calculation of the gPC modes~\eqref{ARGMIN} for an expansion with two basis functions, i.e.~$K=1$. Normalized Legendre, Hermite polynomials and the Haar basis satisfy in this special case the expressions
\begin{equation}\label{ExpressionsFigure}
\begin{aligned}
&\P(\hat{\alpha})
= 
\begin{pmatrix}
\hat{\alpha}_0 & \hat{\alpha}_1 \\
\hat{\alpha}_1 & \hat{\alpha}_0 \\ 
\end{pmatrix},
\quad
\mathcal{R}(
\widehat{\u}
)
=
\begin{pmatrix}
\widehat{\u}_0^2+\widehat{\u}_1^2 \\
2 \, \widehat{\u}_0 \widehat{\u}_1
\end{pmatrix},
\quad
\mathcal{R}^{-1}(
\widehat{\rho}
)
=
\frac{1}{2}
\begin{pmatrix}
\sqrt{ \widehat{\rho}_0 + \widehat{\rho}_1 } 
+
\sqrt{ \widehat{\rho}_0 - \widehat{\rho}_1 }  \\
\sqrt{ \widehat{\rho}_0 + \widehat{\rho}_1 }
-
\sqrt{ \widehat{\rho}_0 - \widehat{\rho}_1 }
\end{pmatrix}, \\
&\text{and}\quad
\quad
\widehat{ \,
	\lVert \u \rVert
	\, }
=
\frac{1}{2}
\begin{pmatrix}
\big| \widehat{\u}_0 + \widehat{\u}_1 \big|  
+
\big| \widehat{\u}_0 - \widehat{\u}_1 \big|  \\
\big| \widehat{\u}_0 + \widehat{\u}_1 \big|  
-
\big| \widehat{\u}_0 - \widehat{\u}_1 \big|  
\end{pmatrix}.
\end{aligned}
\end{equation}

\noindent
The first three panels of Figure~\ref{ARGMIN} illustrate 
these mappings. 
The $x$-axes describe the first mode of the  inverse images. 
The second modes of the inverse images, which account for the random perturbations, are chosen as~$\widehat{\u}_1,\widehat{\rho}_1\in[-1,1]$. This results in different values that are given by two colorbars. The first one on the left hand side states the first mode and the second colorbar states the second mode for the expressions~\eqref{ExpressionsFigure}.   
Furthermore, areas corresponding to the first modes are highlighted by  lines,  those of the second one are illustrated with dashed lines, respectively. 

The first panel shows the mapping~$
\mathcal{R}(\widehat{ \u })$, where the $x$-axis corresponds to the mean~$\widehat{\u}_0\in\mathbb{R}$.   The first mode~$\mathcal{R}(\widehat{ \u })_0\geq 0$ states the positive expected value with respect to the left colorbar 
and the second mode accounts for random perturbations. 
The second panel states the inverse mapping~$\mathcal{R}^{-1}$, which is used to obtain an expression for the square root. 
It is only defined for~$\widehat{\rho}_0 \geq |\widehat{\rho}_1|$. In this case, perturbations are sufficiently small compared to the positive mean~$\widehat{\rho}_0\geq 0$. 
The third panel shows the solution to the minimization problem~\eqref{ARGMIN}.

\vspace{5mm}
\begin{figure}[H]
	\begin{minipage}{0.49\textwidth}
		\begin{center}	
			\  \textbf{(i)}\ \textbf{Galerkin product:} $\mathcal{R}(
			\widehat{\u}
			)$
			
\vspace{-0mm}					
{\small $|\widehat{\u}_1|=1$}
\hfill
{\small $\widehat{\u}_1=1$}

\vspace{-1mm}
			\scalebox{1}{\includegraphics[width=\linewidth]{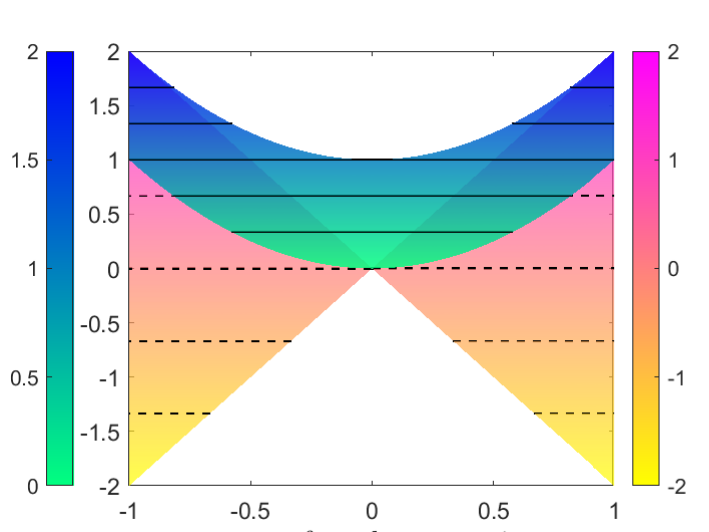}}

\vspace*{-2mm}						
{\small $\widehat{\u}_1=0$} \hfill {\small $\widehat{\u}_1=-1$}	

\vspace*{-1mm}
			{\small mean $\widehat{\u}_0$}	
		\end{center}	
	\end{minipage}
	\hfil
	\begin{minipage}{0.49\textwidth}
		\begin{center}	
	\ \ 	\textbf{(ii)}\	\textbf{Galerkin root:} $\mathcal{R}^{-1}
(\widehat{\rho} )
$

\vspace{-0mm}					
\ \ {\small $|\widehat{\rho}_1|\leq1$}
\hfill
{\small  $\widehat{\rho}_1=1$}

\vspace{-1mm}			
			\scalebox{1}{\includegraphics[width=\linewidth]{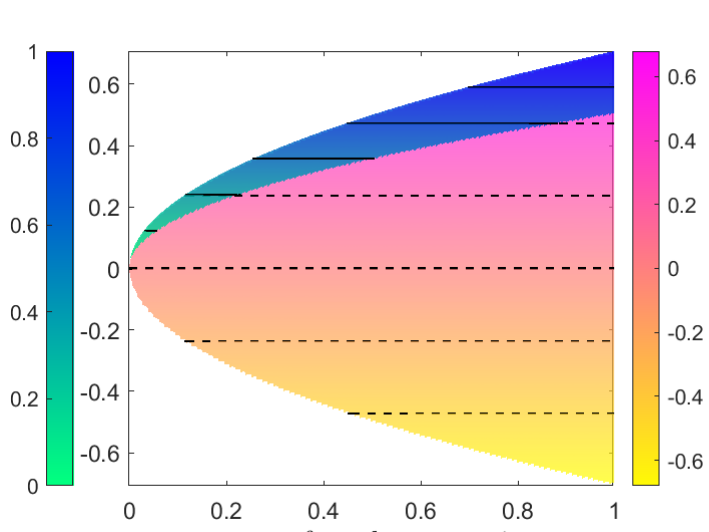}}

\vspace*{-2mm}						
\ \ {\small $\widehat{\rho}_1=0$} \hfill {\small $\widehat{\rho}_1=-1$}	

\vspace*{-1mm}			
			{\small mean $\widehat{\rho}_0=(\widehat{\u}\ast\widehat{\u})_0$}
			
		\end{center}	
	\end{minipage}

	\vspace{6mm}
	
	\begin{minipage}{0.49\textwidth}
		
		\begin{center}	
			\ \textbf{(iii)}\  \textbf{Euclidean norm:} $\widehat{ \,
				\lVert \u \rVert
				\, }$

\vspace{-0mm}			
{\small $|\widehat{\u}_1|=1$}
\hfill
{\small $\widehat{\u}_1=1$}

			\vspace{-1mm}			
			\scalebox{1}{\includegraphics[width=\linewidth]{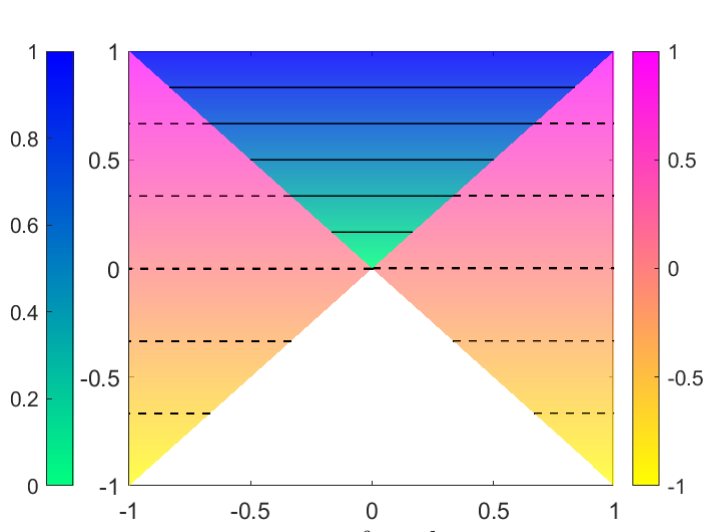}}
			
\vspace*{-2mm}						
{\small $\widehat{\u}_1=0$} \hfill {\small $\widehat{\u}_1=-1$}	

\vspace*{-1mm}
{\small mean $\widehat{\u}_0$}	
			
		\end{center}	
	\end{minipage}
	\hfil
	\begin{minipage}{0.49\textwidth}
		
		\vspace{-2mm}
		\begin{center}
\textbf{(iv)}\			\textbf{Optimization problem~\eqref{ARGMIN}} 
			\vspace*{0.62cm}
			
			\scalebox{1}{\includegraphics[width=\linewidth]{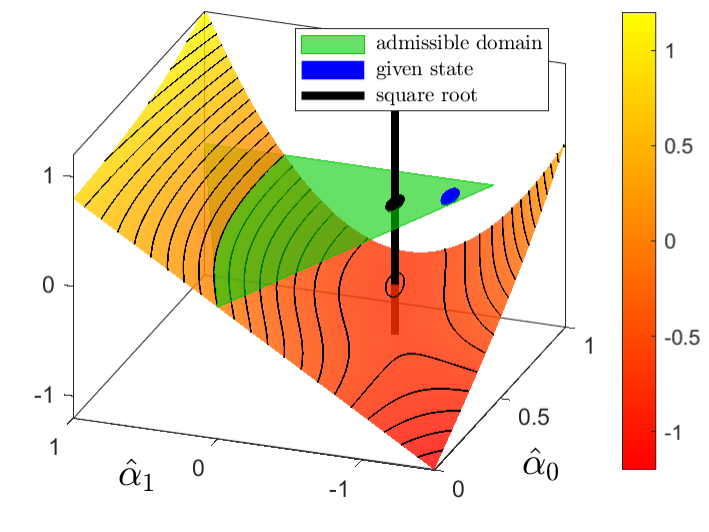}}	
			
		\end{center}
	\end{minipage}
	
	\caption{Illustration of the norm~\eqref{ARGMIN} by means of the expressions~\eqref{ExpressionsFigure},  shown in the panels~(i)~--~(iii). 
	The values stated by the colorbars are as follows: 
	First panel illustrates the modes~$\mathcal{R}(
	\widehat{\u}
	)_0$~(left colorbar) and~$\mathcal{R}(
	\widehat{\u}
	)_1$~(right colorbar) for the Galerkin product, the second panel shows the inverse mapping~$\mathcal{R}^{-1}(
	\widehat{\u}
	)_0$~(left) and~$\mathcal{R}(
	\widehat{\u}
	)_1$~(right). 
	The gPC modes of the Euclidean norm are stated in the third panel, where the first one is shown at the scale of the left colorbar, the second one with respect to the right colorbar. 
	In the first and third panel, the colour blue corresponds to~$|\widehat{\u}_1|=1$ and to~$|\widehat{\rho}_1|\leq1$ in the second panel. Likewise, the other cases are stated above and below the colorbars. 
	Panel~(iv) considers the optimization problem~\eqref{ARGMIN} for a fixed state~$\widehat{ \rho }=(0.89,-0.8)^\T$.}

\end{figure}

The fourth panel states the function 
$\nicefrac{1}{3}\;\hat{\alpha}^\T \P(\hat{\alpha}) \hat{\alpha} - \hat{\alpha}^\T \widehat{ \rho }$. 
It is convex on the green highlighted set, where the condition~$\hat{\alpha}_0 \geq |\hat{\alpha}_1|$ holds. 
For a fixed state~$\widehat{ \rho }=\widehat{\u}^{\ast 2}$ its minimum gives the desired gPC modes that are illustrated in the third panel. This also shows that the optimization problem is well-posed apart from states that are close to zero or result in too large random fluctuations. In those cases, assumption~\eqref{AI} is violated. \\

\noindent
The optimization problem~\eqref{ARGMIN} allows to extend the Hamiltonian~\eqref{RandomHamiltonian} to a stochastic Galerkin formulation. To this end, we assume an arbitrary, but consistent gPC approximation~$\v(\x)$ satisfying
$$
\bigg\lVert
v(\x,\xi)
-
\sum\limits_{k=0}^K
\v_{k}(\x)\phi_{k}(\xi)
\bigg \rVert_{\mathbb{P}}
\rightarrow 0
\quad\text{for}\quad
K\rightarrow\infty. 
$$
The gPC formulation for the random Hamiltonian reads as
\begin{equation}\label{HamiltonianLS}
\H(\widehat{\u},\x)
\coloneqq
\v(\x)
\ast
\widehat{\, \lVert \u \rVert\, }.
\end{equation}
The corresponding stochastic Galerkin formulations to the Hamilton-Jacobi equations~\eqref{HJ} and the hyperbolic conservation laws~\eqref{HC} are
\begin{alignat}{8}
&\partial_t \widehat{\varphi}(t,\x)
&&+
\H\Big( \nabla_{\x} \widehat{\varphi}(t,\x),\x \Big)&&=0,
\tag{\textup{HJ}}
\label{HJ} \\
&\partial_t \widehat{\u}(t,\x)
&&+
\nabla_{\x}
\H\Big(  \widehat{\u}(t,\x),\x \Big)&&=0,
\quad
\widehat{\u}(\x) \coloneqq \nabla_{\x}  \widehat{\varphi}(\x). \label{HC} \tag{\textup{HC}}
\end{alignat}
In the case of deterministic initial data~$\varphi_0(\x)\in\mathbb{R}$ we have 
\begin{align*}
	\widehat{\varphi}_{k}(0,\x) = \varphi_0(\x) \delta_{k,0} \quad \text{and} \quad 
	 \widehat{\u}_{k}(0,\x) = \nabla_{\x} \varphi_0(\x) \delta_{k,0}.
\end{align*}
In the sequel, the hyperbolic form~\eqref{HC} with the Hamiltonian~\eqref{HamiltonianLS}, which corresponds to the random level set equations, is analyzed in terms of hyperbolicity. 

\subsection{Eigenvalue decomposition and hyperbolicity}
We distinguish between the following two formulations, which are formally equivalent provided that the matrix~$\P\big( \v(\x) \big)$ is invertible, but result in different theoretical and numerical solution concepts. \\
 
\noindent
\textbf{Conservative form}\vspace{4mm}

\noindent
$\displaystyle
\begin{aligned}
&\partial_t \widehat{\u}(t,\x)
+
\partial_{x_1}
\widehat{f_1}\Big(\widehat{\u}(t,\x),\v(\x)\Big)
+
\partial_{x_2}
\widehat{f_2}\Big(\widehat{\u}(t,\x),\v(\x)\Big)
=0 \\
& 
\text{with flux functions}\quad
\widehat{f_1}\big(\widehat{\u},\v\big)
=
\begin{pmatrix}
\v \ast
\widehat{ \,
	\lVert \u \rVert
	\, } \\ 0
\end{pmatrix}
\quad\text{and}\quad
\widehat{f_2}\big(\widehat{\u},\v\big)
=
 \begin{pmatrix}
0 \\
\v \ast
\widehat{ \,
	\lVert \u \rVert
	\, }
\end{pmatrix}
\end{aligned} \label{2dLevelSetgPC}
$

\bigskip
\noindent
\textbf{Capacity form}\vspace{4mm}

\noindent
$\displaystyle
\begin{aligned}
&\P\big( \v(\x) \big)^{-1}
\partial_t \widehat{\u}(t,\x)
+
\partial_{x_1}
\widetilde{f_1}\Big(\widehat{\u}(t,\x)\Big)
+
\partial_{x_2}
\widetilde{f_2}\Big(\widehat{\u}(t,\x)\Big)\\
&\qquad\qquad =
-
\P\big( \v(\x) \big)^{-1}
\big(\partial_{x_1}+\partial_{x_2} \big)\v(\x)
\ \ast
\widehat{ \
	\lVert \u(t,\x) \rVert
	\ }
 \\
& 
\text{with flux functions}\quad
\widetilde{f_1}\big(\widehat{\u}\big)
=
\begin{pmatrix}
\widehat{ \,
	\lVert \u \rVert
	\, } \\ 0
\end{pmatrix}
\quad\text{and}\quad
\widetilde{f_2}\big(\widehat{\u}\big)
=
\begin{pmatrix}
0 \\
\widehat{ \,
	\lVert \u \rVert
	\, }
\end{pmatrix}
\end{aligned}
\label{capacityForm}
$

\bigskip

\noindent
The following main theorem investigates the eigenvalue decomposition of both formulations. In particular the capacity form is proven  strongly hyperbolic. The proof is inspired by the approach in~\cite{Epshteyn2021,Epshteyn2022}, where the Jacobian is shown to be similar to a symmetric matrix.\\

\begin{theorem}\label{TheoremHyp}
Let  states~$\widehat{\u}=\big(\uI,\uII\big)^\T\in\mathbb{R}^{2(K+1)}$ and gPC modes~${\v} \in\mathbb{R}^{K+1}$ be given such that for all 
unit vectors~$\vec{n}=(n_1,n_2)^\T$ 
the matrices
$
\widehat{\lambda}(\widehat{\u})
\coloneqq
n_1 \P(\widehat{\,\u_1}) + n_2 \P(\widehat{\,\u_2})
$, 
$
\P\big(\v\big)
$
are invertible and such that property~\eqref{AI} is satisfied. 
For all unit vectors~$\vec{n}=(n_1,n_2)^\T$, the following statements hold:\\

\begin{enumerate}[(i)]
	\item
The \textbf{conservative form} has the Jacobian 	\\

$\displaystyle
\begin{aligned}
\J (\widehat{\u},\v)
&\coloneqq
n_1 \D_{\widehat{\u}} \widehat{f_1}(\widehat{\u},\v)
+
n_2 \D_{\widehat{\u}} \widehat{f_2}(\widehat{\u},\v)\\
&\;=
\Big[
\diag\{1,1\}\otimes 
\A(\widehat{\u},\v)
\Big]
\begin{pmatrix}
n_1 \P\big(\widehat{\,\u_1}\big) & n_1 \P\big(\widehat{\,\u_2}\big) \\
n_2 \P\big(\widehat{\,\u_1}\big) & n_2 \P\big(\widehat{\,\u_2}\big) 
\end{pmatrix} 
\end{aligned}
$\\

$\displaystyle
\begin{aligned}
&\text{which is similar to}
&\widehat{\Lambda}(\widehat{\u},\v) 
&\coloneqq
\begin{pmatrix} \zeros & \zeros \\
\zeros & \widehat{\lambda}(\widehat{\u},\v)
\end{pmatrix} \\
&\text{with}
\quad
&\widehat{\lambda}(\widehat{\u},\v)
&\coloneqq
\A(\widehat{\u},\v)
\Big[
n_1 \P(\widehat{\,\u_1}) + n_2 \P(\widehat{\,\u_2})\Big]\\
&\text{and}
\quad
&\A(\widehat{\u},\v)
&\coloneqq
\P\big(\v\big)  \P\Big( \widehat{ \,
	\lVert \u \rVert
	\, } \Big)^{-1}.
\end{aligned}
$ \\

\medskip

	\item
The \textbf{capacity form} has the Jacobian 	\\

$\displaystyle
\begin{aligned}
\JC (\widehat{\u})
&\coloneqq
n_1 \D_{\widehat{\u}} \widetilde{f_1}(\widehat{\u})
+
n_2 \D_{\widehat{\u}} \widetilde{f_2}(\widehat{\u})\\
&\,=
\bigg[\diag\{1,1\}\otimes
 \P\Big( \widehat{ \,
	\lVert \u \rVert
	\, } \Big)^{-1}
\bigg]
\begin{pmatrix}
n_1 \P\big(\widehat{\,\u_1}\big) & n_1 \P\big(\widehat{\,\u_2}\big) \\
n_2 \P\big(\widehat{\,\u_1}\big) & n_2 \P\big(\widehat{\,\u_2}\big) 
\end{pmatrix} 
\end{aligned}
$\\

and is similar to the symmetric matrix
$$
\widetilde{\Lambda}(\widehat{\u}) 
\coloneqq
\begin{pmatrix} \zeros & \zeros \\
\zeros & \widetilde{\lambda}(\widehat{\u})
\end{pmatrix} 
\ \ \ \text{with}\ \ \
\widetilde{\lambda}(\widehat{\u})
\coloneqq
\P\Big( \widehat{ \,
	\lVert \u \rVert
	\, } \Big)^{-\nicefrac{1}{2}}
\Big[
n_1 \P(\widehat{\,\u_1}) + n_2 \P(\widehat{\,\u_2})\Big]
\P\Big( \widehat{ \,
	\lVert \u \rVert
	\, } \Big)^{-\nicefrac{1}{2}}.
$$

\end{enumerate}

\end{theorem}

\begin{proof} 
We introduce the injective mapping
$$
\mathcal{T}
 \ \ : \ \
\mathbb{R}^{2(K+1)} 
\rightarrow
\mathbb{R}^{K+1} , \ \
\widehat{\u} \mapsto
\widehat{\u}_1^{\ast 2}
+
\widehat{\u}_2^{\ast 2}
$$
which has the Jacobian 
$
\D_{\ui} \mathcal{T}( \widehat{\u}) 
= 2\P(\ui)
$ and we define 
$
\mathcal{R}'( \widehat{\u}) 
\coloneqq
\D_{\widehat{\u}} \mathcal{R}( \widehat{\u}) 
= 2\P\big(\widehat{\u}\big)
$. 
Then, the Jacobian to the gPC modes~\eqref{ARGMIN} reads as
$$
\D_{\ui}
\widehat{ \,
	\lVert \u \rVert
	\, }	
=
\D_{\ui}
\Big[
\mathcal{R}^{-1} \big( \mathcal{T}(\widehat{\u})   \big) 
\Big]
=
\bigg[
\mathcal{R}' \left(
\widehat{ \,
	\lVert \u \rVert
	\, }
\right)
\bigg]^{-1}
\D_{\ui} \mathcal{T}( \widehat{\u}) 
= 
\P\Big( \widehat{ \,
	\lVert \u \rVert
	\, } \Big)^{-1}
\P\big(\ui\big),
$$
which yields the Jacobians~$\J (\widehat{\u},\v)$ and $\JC (\widehat{\u})$. 
Define the matrices
\begin{alignat*}{8}
&\widehat{\textbf{V}}(\widehat{\u},\v)
&&\coloneqq
\Big[\diag\{1,1\}\otimes  \widehat{\lambda}(\widehat{\u},\v)&&\Big]^{-1}
&&\begin{pmatrix}
-\A(\widehat{\u},\v)
\P(\widehat{\,\u_2}) \P(\widehat{\,\u_1})^{-1} 
\A(\widehat{\u},\v)^{-1}
&
 \frac{n_1}{n_2} \indikator \\
\indikator & \indikator
\end{pmatrix},\\
&\widehat{\textbf{V}}(\widehat{\u},\v)^{-1}
&&\coloneqq
\Big[\diag\{1,1\}\otimes \A(\widehat{\u},\v) &&\Big]
&&\begin{pmatrix}
-n_2 \P(\uI) & n_1 \P(\uI) \\
n_2 \P(\uI) & n_2 \P(\uII)
\end{pmatrix}, \quad
\indikator\coloneqq \diag\{1,\ldots,1\}. 
\end{alignat*}

\noindent
Standard computations show
\[
\widehat{\textbf{V}}(\widehat{\u},\v)  \widehat{\textbf{V}}(\widehat{\u},\v)^{-1}=\indikator
\quad \text{and}\quad
\J (\widehat{\u},\v) 
=
\widehat{\textbf{V}}(\widehat{\u},\v)
\widehat{\Lambda}(\widehat{\u},\v)
\widehat{\textbf{V}}(\widehat{\u},\v)^{-1},
\]
which proves the first statement. 
The second statement follows analogously with the~choice~$\v=\UNIT$, i.e.~$\P(\v)=\indikator$, 
and due to the matrix similarities
\[
\JC (\widehat{\u}) 
=
\widehat{\textbf{V}}(\widehat{\u},\UNIT)
\left(
\diag\{1,1\}\otimes
\Bigg[
\P\Big( \widehat{ \,
	\lVert \u \rVert
	\, } \Big)^{-\nicefrac{1}{2}}
\ \, \widetilde{\lambda}(\widehat{\u}) \
\P\Big( \widehat{ \,
	\lVert \u \rVert
	\, } \Big)^{\nicefrac{1}{2}}
\Bigg]
\right)
\widehat{\textbf{V}}(\widehat{\u},\UNIT)^{-1}.
\]

\hfill
\end{proof}

\subsection{Relationship between the conservative and capacity form}\label{SectionCapacity}
We remark that the presented hyperbolicity results hold only local in time around states, where the assumptions of Theorem~\ref{TheoremHyp} are fulfilled, 
i.e.~the matrix~$	\P\big( \widehat{ \,
	\lVert \u \rVert
	\, } \big)$ is strictly positive definite. 
According to~\cite[Th.~2]{Sonday2011} and \cite[Th.~2.1]{H3}, this holds provided that  realizations 
\begin{equation}\label{Condition1}
\GK \left[\widehat{ \,
	\lVert \u \rVert
	\, }\right]\big(\xi(\omega)\big)  
= 
\sum\limits_{k=0}^K
\left(\widehat{ \,
	\lVert \u \rVert
	\, }\right)_{k}
\phi_{K}\big(\xi(\omega)\big) 
> 0
\end{equation}
are almost surely strictly positive. 
Hence, there is a degeneracy of the Jacobian for a vanishing solution~$\widehat{\u}\rightarrow (0,\ldots,0)^\T$. 
This is an expected property that is also observed in the deterministic case~\eqref{DeterministicJacobian}. 
Furthermore, Theorem~\ref{TheoremHyp} only guarantees that the Jacobian~$\JC (\widehat{\u})$ in the capacity form has real eigenvalues and a complete set of eigenvectors, since it is similar to a symmetric matrix. Hyperbolicity of the conservative form is not  guaranteed, because the matrix~$
\A(\widehat{\u},\v)
$ is not necessarily symmetric and the square root may not exist. 
If the velocity is deterministic and positive, i.e.~$v(\x)>0$, we have~$\v(\x)=v(\x)\UNIT$. Then, the matrix
	$$
	\A(\widehat{\u},\v)
	= v\,
	\P\Big( \widehat{ \,
		\lVert \u \rVert
		\, } \Big)^{-1}
$$
is 	symmetric and positive definite. 
In general however, the Jacobian~$\J (\widehat{\u},\v)$ may have complex eigenvalues. An example is given in the appendix. 
Then, the capacity form has to be considered. 

Note that this form naturally occurs e.g.~in the derivation of conservation laws if the flux of a quantity is defined in terms of a quantity that is \emph{not} conserved~\cite[Sec.~2.4, Sec.~6.16]{Leveque}. We consider an arbitrary spatial domain~$C$, where the matrix~$\P\big(\v(\x)\big)$ is invertible. 
Then, the integral forms to the two-dimensional level set equations read as follows: \\

\noindent
\textbf{Conservative form} \vspace{2mm}

\noindent
$\displaystyle
\begin{aligned}
\frac{\textup{d}}{\textup{d} t}
\int_{C} 
\widehat{\u} (t,\x)
\d x 
=
&-\int_{\partial C}
n_1(s) 
\widehat{f_1}\Big(\widehat{\u}\big(t,\x(s)\big), \v\big(\x(s)\big) \Big)
+
n_2(s)
\widehat{f_2}\Big(\widehat{\u}\big(t,\x(s)\big), \v\big(\x(s)\big)\Big)
\d s 
\end{aligned}
$

\bigskip\smallskip

\noindent
\textbf{Capacity form} \vspace{2mm}

\noindent
$\displaystyle
\begin{aligned}
\frac{\textup{d}}{\textup{d} t}
\int_{C} 
\P\big(\v(\x)\big)^{-1}
\widehat{\u} (t,\x)
\d \x 
=
&-\int_{\partial C}
n_1(s) 
\widetilde{f_1}\Big(\widehat{\u}\big(t,\x(s)\big)\Big)
+
n_2(s)
\widetilde{f_2}\Big(\widehat{\u}\big(t,\x(s)\big)\Big)
\d s \\
&-
\int_{C} 
\P\big(\v(\x)\big)^{-1} \;
\big(\partial_{x_1}+\partial_{x_2} \big)\v(\x) \;
\ast
\widehat{ \
	\lVert \u(t,\x) \rVert
	\ }
\d \x. 
\end{aligned}
$

\bigskip
\noindent


\section{Hyperbolicity preserving finite-volume discretization}\label{SectionFV}
To improve and faciliate the readability, we present the numerical discretization in one spatial dimension and we assume a constant velocity~$\v(x)=\v$.  Then, the conservative and capacity form read as
\begin{alignat*}{8}
&
&&\partial_t \widehat{\u}(t,x)
+
\partial_{x}
\widehat{f}\Big(\widehat{\u}(t,\x),\v\Big)
&&=0
\ \ \text{for} \ \
\widehat{f}\big(\widehat{\u},\v\big)
&&=
\v \ast
\widehat{ \,
	| \u |
	\, }, \ \,
&&\JoneDim \big(\widehat{\u},\v\big)
&&=
\P\big(\v\big) 
\JtildeoneDim \big(\widehat{\u} \big) \\
&\P(\v)^{-1}
&&\partial_t \widehat{\u}(t,x)
+
\partial_{x}
\widetilde{f}\Big(\widehat{\u}(t,\x)\Big)
&&=0
\ \ \text{for} \ \
\widetilde{f}\big(\widehat{\u}\big)
&&=
\widehat{ \,
	| \u |
	\, },
\ \,
&&\JtildeoneDim \big(\widehat{\u} \big)
&&=
 \P\big( \widehat{ \,
	| \u |
	\, } \big)^{-1}
\P\big( \widehat{\u} \big),
\end{alignat*}
where~$
\JoneDim \big(\widehat{\u},\v\big)
\coloneqq
\D_{\widehat{\u}}
\widehat{f}\big(\widehat{\u},\v\big)
$ 
and~$
\JtildeoneDim \big(\widehat{\u}\big)
\coloneqq
\D_{\widehat{\u}}
\widetilde{f}\big(\widehat{\u}\big)
$ 
denote the corresponding Jacobian. 
The interested reader finds  the two-dimensional case with  space-varying velocity in the appendix. 
The spatial domain is discretized in equidistant grid cells ${C_{j}
\coloneqq
\big(x_{j-\nicefrac{1}{2}}, x_{j+\nicefrac{1}{2}} \big)
\subset \mathbb{R} }$ with volume~$
\DxI=x_{j+\nicefrac{1}{2}}- x_{j-\nicefrac{1}{2}} >0
$  
centered at~$x_{j}$. The cell averages at time~$t^n\in\mathbb{R}^+_0$ are denoted by~$\ubar_{j}^n$ and~$\wbar_{j}^n$.

\subsection{Conservative form}\label{SecDiscretizationConservative}
Under the assumption of a real  spectrum,  which has been deduced in Theorem~\ref{TheoremHyp}, 
the \textbf{local Lax-Friedrichs flux} reads  as
\begin{equation}\label{NumFluxConserv}
	\NumFlux \big(\ubarL,\ubarR,\v\big) 
	\coloneqq 
	\frac{1}{2} \Big( \widehat{f}\big(\ubarL,\v\big)+\widehat{f}\big(\ubarR,\v\big) \Big) 
	- 
	\frac{1}{2}
	\max\limits_{q=\ell,r} \Big\{ 
	\sigma\Big\{ 
	\JoneDim
	\big( \ubarK,\v\big) \Big\}
	\Big\}
	\big( \ubarR - \ubarL \big),
\end{equation}
where~${
	\sigma\{ \cdot \}
}$ denotes the spectral radius. 
Then, a first-order finite-volume discretization reads as
\begin{equation}\label{conservativeUpdate}
\ubar_{j}^{n+1}
=
\ubar_{j}^{n}
- \frac{\Delta t}{\Delta x}
\bigg[
\NumFlux\left(\ubar_{j}^{n},\ubar_{j+1}^{n},\v \right)
		-
\NumFlux\left(\ubar_{j-1}^{n},\ubar_{j}^{n},\v\right)
\bigg].
\end{equation}

\noindent
Provided that the CFL-condition
\begin{equation}\label{CFLconservative}
	\max\limits_{j\in\mathbb{N}} \Big\{ 
\sigma\left\{ 
\JoneDim
\big( \ubar_j^n,\v\big) \right\} \Big\}
	\frac{\Delta t}{\Delta x}
	<1
\end{equation}
holds, the  discrete solution converges to the weak entropy solution~\cite{Leveque}.

\subsection{Capacity form}\label{SectionNumericsCapacity}
We have already remarked that the conservative form is  hyperbolic if the velocity is deterministic. 
Then, the CFL-condition ensures that the dependency of the true
solution is within the numerical range~\cite[Sec.~4.4]{Leveque}, which results in a convergent numerical scheme. 
Applying the stochastic Galerkin method to random velocities, however, results in additional waves that influence the numerical dependence on the solution. 
The idea of the hyperbolicity preserving discretization is to introduce a \emph{non-uniform effective grid} to capture the numerical dependence in an appropriate way. 
Using the orthogonal eigenvalue decomposition~$
\P(\v) 
=
\VP(\v)  \DP(\v)  \VP(\v) ^\T, 
$ 
the one-dimensional capacity form is equivalent to
\begin{equation}\label{ToyProblem}
	\begin{aligned}
&\DP(\v)^{-1} \partial_t \w(t,x)
+
\partial_{x}
\Feffective \big(\w(t,x),\v \big) 
=0
\quad\text{with}\\
&\w \coloneqq \VP(\v)^\T \widehat{\u} 
\quad\text{and}\quad
\Feffective \big(\w,\v \big)
\coloneqq
\VP(\v)^\T \,
\widetilde{f}\big( \VP(\v) \w \big).
\end{aligned}
\end{equation}
The $k$-th component of the vector valued cell average~$	\wbar_{j}^{n}\in\mathbb{R}^{K+1}$   is given by the recursion
\begin{equation}\label{NonEquidistant}
	\begin{aligned}
		&\left(\wbar_{j}^{n+1} \right)_k
		=
		\left(\wbar_{j}^{n} \right)_k
		-
		\frac{
			{\Delta t}
		}{\Delta \tilde{v}_{k}}
		\bigg[
		\NumFluxC_{k}^{(\textup{gLF})} \left(\wbar_{j}^{n},\wbar_{j+1}^{n},\v \right)
		-
		\NumFluxC_{k}^{(\textup{gLF})} \left(\wbar_{j-1}^{n},\wbar_{j}^{n},\v \right)
		\bigg] \\
		&\text{with the effective volume}
		\quad  \Delta \tilde{v}_{k} \coloneqq 
		\frac{\Delta x}{\DP_k (\v)}.
	\end{aligned}
\end{equation}

\noindent
Here, the componentwise \textbf{global Lax-Friedrichs flux}  takes the effective volume into account  and  reads  as
\begin{equation}\label{globalLF}
	\NumFluxC_{k}^{(\textup{gLF})} \big(\wbarL,\wbarR,\v\big) 
	\coloneqq 
	\frac{1}{2} \Big( \Feffective_{k}\big(\wbarL,\v\big)+\Feffective_{k}\big(\wbarR,\v\big) \Big) 
	- 
	\frac{1}{2}
\frac{\Delta \tilde{v}_{k}}{\Delta t}
	\big( \wbarR - \wbarL \big).
\end{equation}
We observe from the diffusion part of the componentwise numerical flux function~\eqref{globalLF} that the resulting CFL-condition is 
\begin{equation*} 
\vMAX \
	\max\limits_{j\in\mathbb{N}} \Big\{ 
	\sigma\left\{ 
	\JtildeoneDim
	\big( \ubar_j^n \big) \right\}
	\Big\}
	\frac{\Delta t}{\Delta x}
	<1
\quad \text{for} \quad
\vMAX
	\coloneqq
	\max\limits_{ k =0,\ldots,K } \Big\{ \big|
\DP_k (\v)
\big| \Big\}
\quad \text{and} \quad
\ubar_j^n
=
\VP(\v)
\wbar_j^n.
\end{equation*}
Following \cite[Sec.~6.17]{Leveque}, we view the scheme~\eqref{NonEquidistant} as a \emph{non-uniform effective space discretizations}~$\Delta \tilde{v}_{k}\in\mathbb{R} $ that is related by a coordinate transform to a \emph{uniform computational domain}~$\Delta x \in \mathbb{R}^+$. 
At this point, we emphasize that we do not make any assumption on the direction of the random velocities and hence on the sign of the effective volume. This is a desired result, since the solution to the level set equations behaves symmetric in terms of the sign of the velocity. 
We introduce  for each component~$\w_k$ the change of variables~$
\varphi_k(x) =\tilde{v}
$ 
satisfying~$
\varphi_k'(x)\DP_k \big(\v(x)\big)=1
$ 
and~$
\w_k(t,x)
=
\wtilde_k \big(t,\varphi_k(x)\big)
$. 
This yields
$$
\int_{x_{j-\nicefrac{1}{2}}}^{x_{j+\nicefrac{1}{2}}}
\w_k\big(t,x\big)
\DP_k \big(\v(x)\big)^{-1}
\d x
= 
\int_{x_{j-\nicefrac{1}{2}}}^{x_{j+\nicefrac{1}{2}}}
\wtilde_k \big(t,\varphi_k(x)\big)
\varphi_k'(x)
\d x
=
\int_{\tilde{v}_{j-\nicefrac{1}{2},k}}^{\tilde{v}_{j+\nicefrac{1}{2},k}}
\wtilde_k \big(t,\tilde{v}\big)
\d \tilde{v}
$$
for~$\tilde{v}_{j\pm\nicefrac{1}{2},k} \coloneqq \varphi_k (x_{j\pm\nicefrac{1}{2},k} ) $. Therefore, the capacity form 
$$
\frac{\textup{d}}{\d t}
\int_{x_{j-\nicefrac{1}{2}}}^{x_{j+\nicefrac{1}{2}}}
\w_k\big(t,x\big)
\DP_k \big(\v(x)\big)^{-1}
\d x
=
\Feffective_k\Big(
\w\big(t,x_{j-\nicefrac{1}{2}}\big),\v
\Big)
-
\Feffective_k\Big(
\w\big(t,x_{j+\nicefrac{1}{2}}\big),\v
\Big)
$$
on the computational domain can be rewritten into the conservative form 
\begin{equation}\label{ConservativeForm}
	\frac{\textup{d}}{\d t}
	\int_{\tilde{v}_{j-\nicefrac{1}{2},k}}^{\tilde{v}_{j+\nicefrac{1}{2},k}}
	\wtilde_k \big(t,\tilde{v}\big)
	\d \tilde{v}
	=
	\Feffective_k\Big(
	\wtilde\big(t,\tilde{v}_{j-\nicefrac{1}{2},k}\big),\v
	\Big)
	-
	\Feffective_k\Big(
	\wtilde\big(t,\tilde{v}_{j+\nicefrac{1}{2},k}\big),\v
	\Big).
\end{equation}

\noindent
Figure~\ref{FigureCapacity} illustrates the previous analysis. 
The left panel shows the uniform computational domain in black. The propagation of the continuous solution is exemplified by a blue and red, dashed arrow. 
The case when the true solution (red arrow) is not within the numerical range may happen for the conservative form if the Jacobian~$\JoneDim \big(\widehat{\u},\v\big)$ has complex eigenvalues. 
In contrast, the capacity-form differencing~scheme leads to a stable discretization as illustrated by the blue arrow. 
The right panel shows the case when the $k$-th component has a fast speed~$\DP_k \big(\v(x)\big)$, which results in a faster propagation of the true solution~(blue arrow). 
This requires a larger effective volume
$ \Delta \tilde{v}_{j,k} 
=
\tilde{v}_{j+\nicefrac{1}{2},k} - \tilde{v}_{j-\nicefrac{1}{2},k}$. 
Furthermore, the effective volume satisfies the conservative form~\eqref{ConservativeForm}.

\begin{figure}[h]

\begin{tikzpicture}[scale=1.5,cap=round]
	
	
	\begin{scope}
		
		\tikzstyle{important line}=[very thick]
		
		
		\draw[style=important line] (-0.2,0) -- (3.2,0);
		\draw[style=important line] (-0.2,1.2) -- (3.2,1.2);
		
		\draw[style=important line,] (0,-0.2) -- (0,1.4);
		\draw[style=important line,] (1,-0.2) -- (1,1.4);
		\draw[style=important line,] (2,-0.2) -- (2,1.4);
		\draw[style=important line,] (3,-0.2) -- (3,1.4);
		
		\draw (1,-0.4) -- (1,-0.4) node {$x_{j-\nicefrac{1}{2}}$};
		\draw (2,-0.4) -- (2,-0.4) node {$x_{j+\nicefrac{1}{2}}$};
		
		
		\draw[thick,myblue,decoration={markings,mark=at position 1 with	{\arrow[scale=2,>=stealth]{>}}},postaction={decorate}] (1,0) -- (1.9,1.16);
		\draw[thick,dashed,myred,decoration={markings,mark=at position 1 with	{\arrow[scale=2,>=stealth]{>}}},postaction={decorate}] (1,0) -- (2.2,1.16);

		\draw (-0.45,0.6) -- (-0.45,0.6) node {time};
		\draw (-0.45,1.2) -- (-0.45,1.2) node {$t^{n+1}$};
		\draw (-0.45,0) -- (-0.45,0) node {$t^{n}$};
		
		\draw (1.5,1.5) -- (1.5,1.5) node {\textbf{computational domain}};

		
	\end{scope}
	
	
	\begin{scope}[xshift=4.8cm]
		
		\fill [white!90!gruen] (0.5,0) rectangle (2.5,1.2);
		
		\tikzstyle{important line}=[very thick]
		
		
		\draw[style=important line] (-0.2,0) -- (3.2,0);
		\draw[style=important line] (-0.2,1.2) -- (3.2,1.2);
		
		\draw[style=important line,] (0,-0.2) -- (0,1.4);
		\draw[style=important line,] (0.5,-0.2) -- (0.5,1.4);
		\draw[style=important line,] (2.5,-0.2) -- (2.5,1.4);
		\draw[style=important line,] (3,-0.2) -- (3,1.4);
		
		\draw (0.5,-0.4) -- (0.5,-0.4) node {$\tilde{v}_{j-\nicefrac{1}{2},k}$};
		\draw (2.5,-0.4) -- (2.5,-0.4) node {$\tilde{v}_{j+\nicefrac{1}{2},k}$};
		
		

		
		\draw[thick,myblue,decoration={markings,mark=at position 1 with	{\arrow[scale=2,>=stealth]{>}}},postaction={decorate}] (0.5,0) -- (2.4,1.16);

		\draw (1.5,1.5) -- (1.5,1.5) node {\textbf{effective volume}};
		
		\draw (-0.45,0.6) -- (-0.45,0.6) node {time};
		\draw (-0.45,1.2) -- (-0.45,1.2) node {$t^{n+1}$};
		\draw (-0.45,0) -- (-0.45,0) node {$t^{n}$};

		\draw (1.5,0.6) -- (1.5,0.6) node {$\displaystyle \color{gruen}
			\int\limits_{\tilde{v}_{j-\nicefrac{1}{2},k}}^{\tilde{v}_{j+\nicefrac{1}{2},k}}
			\wtilde_k \big(t,\tilde{v}\big)
			\d \tilde{v}
			$};

	\end{scope}
	
\end{tikzpicture}

	\caption{Uniform computational domain~(left): Blue arrow illustrates the capacity-form differencing scheme, red arrow a wrong estimation of characteristic speeds. Non-uniform effective volume~(right): Green area shows the conserved quantity~\eqref{ConservativeForm}, the blue arrow exemplifies a solution that is propagated with a faster (random) speed, which needs a larger effective volume~$\Delta \tilde{v}_{j,k}$.}
	\label{FigureCapacity}
\end{figure}
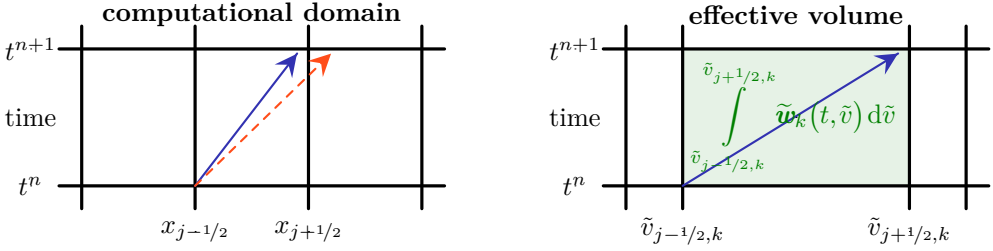

Finally, we remark that the 
hyperbolicity preserving scheme~\eqref{NonEquidistant}  can be implemented in a similar way as the conservative update~\eqref{conservativeUpdate}. Furthermore, the assumption~$\DP_k \big(\v(x)\big) \neq 0$ is~\emph{not} required if the update is defined as
\begin{align*}
\wbar_{j}^{n+1}
&\; = \,
\wbar_{j}^{n} 
-
\frac{
	{\Delta t}
}{\Delta x}
\bigg[
\NumFluxC \left(\wbar_{j}^{n},\wbar_{j+1}^{n},\v \right)
-
\NumFluxC \left(\wbar_{j-1}^{n},\wbar_{j}^{n},\v \right)
\bigg], \\
	\NumFluxC \big(\wbarL,\wbarR,\v\big) 
&	\coloneqq 
	\frac{1}{2} \Big( \widetilde{F}\big(\wbarL,\v\big)+\widetilde{F}\big(\wbarR,\v\big) \Big) 
	- 
	\frac{1}{2}
	\max\limits_{q=\ell,r} \Big\{ 
	\sigma\Big\{ 
	\JtildeoneDim
	\big( \wbarK ) \Big\}
	\Big\}
	\big( \wbarR - \wbarL \big), \\
\widetilde{F}\big(\wbar,\v \big)
&\coloneqq
\DP \big(\v(x)\big) 
\VP(\v)^\T \,
\widetilde{f}\big( \VP(\v) \wbar \big).
\end{align*}

\subsection{Computing the quantile of the perturbed level set}\label{SectionComputeQuantile}
For  convenience of the reader we follow the elementary argument in~\cite[Sec.~2]{Jin1998} that shows how the desired solution~$\widehat{\varphi}$ is obtained by the hyperbolic form~\eqref{HC} with solution~$\widehat{\u}=\nabla_{\x}\hat{\varphi}$, considered in Theorem~\ref{TheoremHyp} and Section~\ref{SectionFV} in terms of a zero viscosity limit. 
To this end, we assume 
~$\widehat{\u}\in L^\infty\big( (0,T)\times\mathbb{R}^d;\mathbb{R}^{d(K+1) }\big) $ 
is a weak viscosity solution to the hyperbolic form~\eqref{HC}. 
Then, there exists a smooth solution~$
\widehat{\u}^{\varepsilon}
\in
 C^2\big( (0,T)\times\mathbb{R}^d;\mathbb{R}^{d(K+1)}\big) 
$ to the viscous Cauchy problem
\begin{equation}\label{visciousForm}
\partial_t \widehat{\u}^{\varepsilon}(t,\x)
+
\nabla_{\x}
\H\Big(  \widehat{\u}^{\varepsilon}(t,\x),\x \Big)
=
\varepsilon \Delta_{\x}  \widehat{\u}^{\varepsilon}(t,\x)
\quad\text{for}\quad
\widehat{\u}^{\varepsilon}(0,\x)
=
\mathcal{I}_{\widehat{\u}}(\x),
\end{equation}
where we assume for simplicity smooth initial data. Furthermore, there exists a constant~$c>0$ such that for any~$m<\infty$ we have
\begin{equation}\label{HEAT0}
\big\lVert \widehat{\u}^{\varepsilon} \big\rVert_{L^\infty} <c
\quad\text{for all}\quad
\varepsilon>0
\quad\text{and}\quad
\widehat{\u}^{\varepsilon}
\overset{L^\infty}{\rightharpoonup}
\widehat{\u}\in L^m
\quad\text{in the limit}\quad
\varepsilon\rightarrow0.
\end{equation}

\noindent
We consider the heat equations
\begin{alignat}{8}
&\partial_t \widehat{\varphi}^{\varepsilon}(t,\x)
&&-
\varepsilon \Delta_{\x}  \widehat{\varphi}^{\varepsilon}(t,\x)
&&=-
&&\H\Big(  \widehat{\u}^{\varepsilon}(t,\x),\x \Big)
&&\quad\text{for}\quad
\widehat{\varphi}^{\varepsilon}(0,\x)
&&=
&&\mathcal{I}_{\widehat{\varphi}}(\x), \label{HEAT1}\\
&\partial_t \widehat{\varphi}_{\x}^{\varepsilon}(t,\x)
&&-
\varepsilon \Delta_{\x}  \widehat{\varphi}_{\x}^{\varepsilon}(t,\x)
&&=-\nabla_{\x}
&&\H\Big(  \widehat{\u}^{\varepsilon}(t,\x),\x \Big)
&&\quad\text{for}\quad
\widehat{\varphi}_{\x}^{\varepsilon}(0,\x)
&&=\nabla_{\x}
&&\mathcal{I}_{\widehat{\varphi}}(\x), \label{HEAT2}
\end{alignat}
where the states~$
\widehat{\varphi}_{\x}
\coloneqq 
\nabla_{\x}
\widehat{\varphi}
$ 
are obtained by differentiating equation~\eqref{HEAT1}. 
Since equation~\eqref{HEAT1} is a system of decoupled heat equations, a maximum principle can be applied to each component. 
Hence, there exist constants~$C_{{k}}(T)>0$ such that the bound $
\big|\widehat{\varphi}_{{k}} \big|
 \leq 
C_k(T)
$ holds
for all $t\in (0,T]$, which are independent for~$\varepsilon>0$ due to  property~\eqref{HEAT0}. Subtracting the viscous form~\eqref{visciousForm} from the heat equation~\eqref{HEAT2} yields 
$$
\partial_t 
\Big( \widehat{\varphi}_{\x}^{\varepsilon}-\widehat{\u}^{\varepsilon} \Big)(t,\x)
-\varepsilon
\Delta_{\x}
\Big( \widehat{\varphi}_{\x}^{\varepsilon}-\widehat{\u}^{\varepsilon} \Big)(t,\x)
=0
\quad\text{for}\quad
\Big( \widehat{\varphi}_{\x}^{\varepsilon}-\widehat{\u}^{\varepsilon} \Big)(0,\x)
=0.
$$
The uniqueness of this Cauchy problem yields the relation~$
\widehat{\varphi}_{\x}^{\varepsilon}(t,\x)=\widehat{\u}^{\varepsilon} (t,\x)
$ and~equation~\eqref{HEAT1} reads as
$$
\partial_t \widehat{\varphi}^{\varepsilon}(t,\x)
+
\H\Big( \nabla_{\x} \widehat{\varphi}^{\varepsilon}(t,\x),\x \Big)
=
\varepsilon \Delta_{\x}  \widehat{\varphi}^{\varepsilon}(t,\x).
$$
Since~$\widehat{\varphi}^{\varepsilon}$ is uniformly bounded in $W^{1,\infty}$, there exists a unique viscosity solution~$
\widehat{\varphi}
\in W^{1,\infty}$ that satisfies~$\nabla_{\x}\widehat{\varphi}=\widehat{\u}$ almost everywhere. \\

\begin{remark}
We refer the interested reader to~\cite{Jin1998,Kruskov1975,Crandall1983} for more details on the relation in terms of viscosity solutions. 
The converse statement in the deterministic case, i.e.~${\u}$ is the vanishing viscosity solution if~$\varphi$ is a viscosity solution, can be shown by exploiting the convexity of the deterministic~Hamiltonian.  	

The projected Hamiltonian~$\H$, however, is not a scalar function. Therefore, the converse statement is not straightforward and left here as an open problem. 
In this work the analysis is carried out in the variable~$\widehat{\u}$. Hence, it is only relevant that~$\widehat{\varphi}$ is a viscosity solution for given~$\widehat{\u}$.
	\end{remark}

\medskip
\noindent
Finally, we note that the matrices~$\mathcal{M}_k$ in the definition~\eqref{2ndMoment} can  be exactly precomputed by Gaussian quadrature. 
According to~\cite[Sec.~5.1]{GersterJCP2019}, 
the positive definiteness assumption~\eqref{AI}, 
which is required for the  optimization problem~\eqref{ARGMIN},  may be efficiently verified by evaluating condition~\eqref{Condition1} at  quadrature nodes~$\xi^{(1)},\ldots,\xi^{(Q)}$. Namely, the inequalities
\begin{equation}\label{QuadratureNodes}
	\GK \left[\widehat{ \,
		\lVert \u \rVert
		\, }\right]\left(\xi^{(q)}\right) >0 
	\quad \text{must hold  at all nodes} \quad
	\xi^{(1)},\ldots,\xi^{(Q)}. 
\end{equation}

\section{Computational experiments}
The theoretical results are illustrated numerically for a Riemann problem with deterministic initial values
$$
\u(0,x)
=
\begin{cases}
	-1 & \text{for} \ \ x<0, \\
	\hphantom{+} 1 & \text{for} \ \ x>0
\end{cases}
\quad  \text{and uniformly distributed  velocity}\quad
v(\xi) \sim \mathcal{U}\big( \nicefrac{1}{2},    \nicefrac{3}{2} \big).
$$
The exact solution reads as
\begin{equation}\label{reference}
	\u(t,x,\xi)
	=
	\begin{cases}
		-1 & \text{for} \ \ x<t v(\xi), \\
		\hphantom{+} 1 & \text{for} \ \ x> t v(\xi)
	\end{cases}
	\quad\text{and}\quad
	\u(t,x,\xi)=0 
	\quad\text{otherwise}.
\end{equation}
As remarked in~Section~\ref{SectionCapacity}, the presented stochastic Galerkin formulations are only obtained by a well-posed optimization problem provided that assumption~\eqref{AI} holds, which is computationally verified by the inequality~\eqref{Condition1}. 
In the other case, regularization techniques are needed. 
More precisely, we consider a treshold parameter~$0<\treshold\ll 1$ and define the set
$$
\mathbb{T}
\coloneqq
\bigg\{
x\in\mathbb{R} \ \Big| \
\GK \left[\widehat{ \,
	\lVert \u \rVert
	\, }\right](\xi)  
< \treshold
\bigg\},
$$
where the optimization problem~\eqref{ARGMIN} is not necessarily well-posed. 
In this set we use the gPC expansion with two modes, i.e.~the truncation~$K=1$. Then, the gPC modes of the norm and the Jacobian are explicitly obtained by the equality~\eqref{ExpressionsFigure} and read as
\begin{align*}
	&\widehat{ \,
		\lVert \u \rVert
		\, }
	=
	\frac{1}{2}
	\begin{pmatrix}
		\big| \widehat{\u}_0 + \widehat{\u}_1 \big|  
		+
		\big| \widehat{\u}_0 - \widehat{\u}_1 \big|  \\
		\big| \widehat{\u}_0 + \widehat{\u}_1 \big|  
		-
		\big| \widehat{\u}_0 - \widehat{\u}_1 \big|  
	\end{pmatrix}
	\quad
	\text{with}
	\quad
	\JtildeoneDim \big(\widehat{\u} \big)
	=
	\D_{\widehat{\u}}
	\widehat{ \,
		\lVert \u \rVert
		\, }
	=
	\begin{pmatrix}
		s_1(\widehat{\u}) & s_2(\widehat{\u}) \\
		s_2(\widehat{\u}) & s_1(\widehat{\u}) 
	\end{pmatrix}, \\
	&  \text{for}\ \ 
	s_1(\widehat{\u}) 
	\coloneqq
	\sign\big(
	\widehat{\u}_0+ \widehat{\u}_1
	\big)
	+
	\sign\big(
	\widehat{\u}_0- \widehat{\u}_1
	\big)
	\ \ \text{and} \ \
	s_2(\widehat{\u}) 
	\coloneqq
	\sign\big(
	\widehat{\u}_0+ \widehat{\u}_1
	\big)
	-
	\sign\big(
	\widehat{\u}_0- \widehat{\u}_1
	\big).
\end{align*}
Likewise to the deterministic case~\eqref{DeterministicJacobian}, the derivative
of the norm is understood as the generalized gradient~\cite{Clarke1990,Frankowska1989,LipschitzFlux2001}.   
It is important to note that the reduction of the gPC truncation to~$K=1$ decreases the spectrum that enters the local Lax-Friedrichs flux. Hence, this approach~\emph{does not ensure} enough numerical viscosity for a stable finite-volume discretization. To this end, the local Lax-Friedrichs flux is replaced by the global flux, however, only for states that belong to the domain~$\mathbb{T}$. In the other case, the numerical discretizations of Section~\ref{SectionFV} are used. 
Finally, we emphasize that the conservative form is only included for comparison, since hyperbolicity is not necessarily guaranteed as the counterexample in~Appendix~\ref{SecCounterExample} shows.

\begin{figure}[H]
	\begin{minipage}{0.49\textwidth}
		\begin{center}	
			{\textbf{capacity form}} 			
			\vspace{-1mm}
			\scalebox{0.85}{\includegraphics[width=\linewidth]{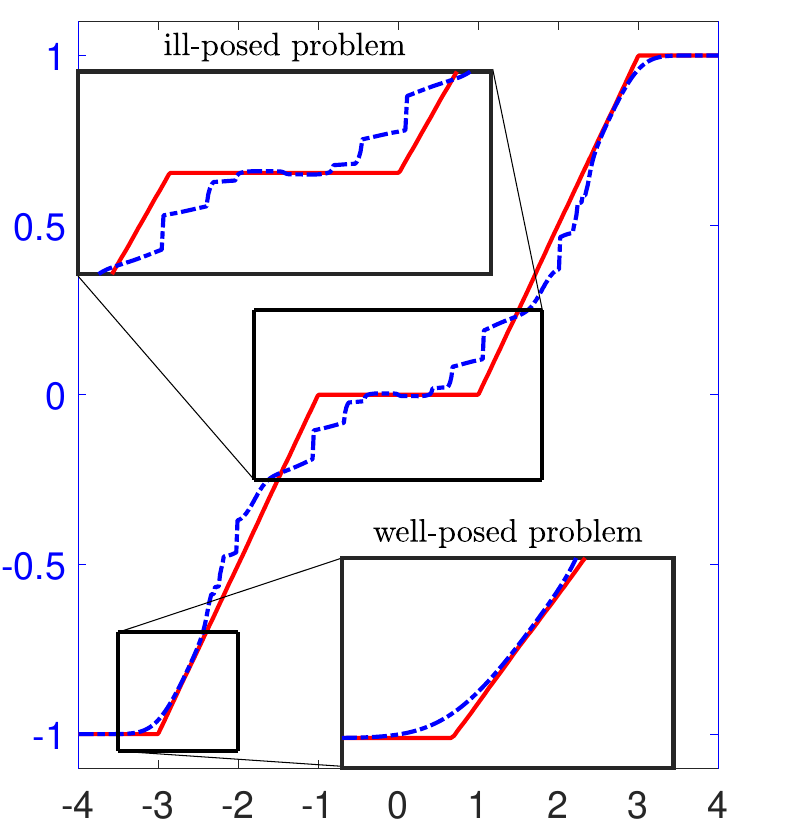}}			
			\vspace*{-1mm}
			spatial domain
		\end{center}	
	\end{minipage}
	\hfil
	\begin{minipage}{0.49\textwidth}
		\begin{center}	
			{\textbf{conservative form}} 			
			\vspace{-1mm}
			\scalebox{0.85}{\includegraphics[width=\linewidth]{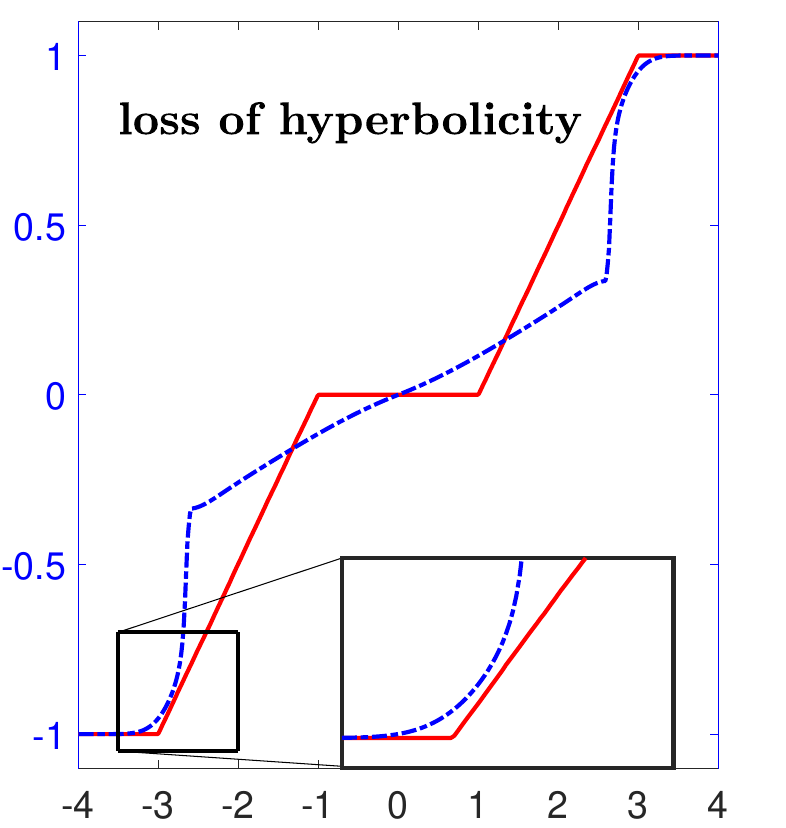}}			
			\vspace*{-1mm}
			spatial domain
		\end{center}	
	\end{minipage}

	\caption{Numerical results for the conservative form presented in Section~\ref{SecDiscretizationConservative} and capacity form in~Section~\ref{SectionNumericsCapacity}. The spatial discretization is~$\Delta x = 2^{-6}$, the final time is~$t=1$ and the CFL number is chosen as~$\textup{CFL}=0.95$. The solid red line states the mean of the reference solution and the numerical solutions to the stochastic Galerkin systems are blue dotted. The zoom in the upper left corner highlights the domain close to the random level-set where oscillations occur due to the ill-posed optimization problem~\eqref{ARGMIN}. The zoom in the lower right corner shows a region where the assumptions in Theorem~\ref{TheoremHyp} hold.}
	\label{FigureNumerics}
	
\end{figure}

Computational results are given in~Figure~\ref{FigureNumerics} for the capacity form~(left panel) and conservative form~(right panel). Therein, the reference solution~\eqref{reference} at time~$t=1$ is stated in terms of the mean, which is plotted as red solid line. 
The mean of the stochastic Galerkin formulation  with truncation~$K=6$, i.e.~the gPC mode~$\widehat{\u}_0(1,x)$, is the blue dotted line.  

The numerical experiments, which are shown in the left panel for the capacity form,  indicate that the capacity-form differencing scheme leads to minior computational errors apart from the random level-set. In regions where the optimization problem is not necessarily well-posed, however, significant errors occur. This shows the need for  regularization techniques, as e.g.~used in~\cite{Alldredge2012}, or piecewise ansatz functions that exploit the local structures in contrast to globally defined Legendre polynomials. 

Since the uncertainty arisis in this problem from the velocity, the conservative form can be non-hyperbolic. Hence, the numerical solution of a finite-volume method, which is shown in the right panel, differs significantly from the reference solution. This explains the need for considering a numerical discretization of the capacity form.

\section{Summary and outlook}
We have analyzed the mathematical description of random zero-level sets that describe moving interfaces. The underlying equations are in  Hamilton-Jacobi form, but are equivalent to a hyperbolic form in the sense of viscosity solutions, which have been analyzed in terms of hyperbolicity. 

Figure~\ref{FigureAdaptivity} summarizes the theoretical results. It  distinguishes between the red area (left), which corresponds to solutions satisfying~$
\GK \big[\widehat{ \,
	\lVert \u \rVert
	\, }\big]\big(\xi(\omega)\big) 
 >0$, and the random interface~(right, blue), where realizations 
may be non-positive. 
In the left, red region  the optimization problem~\eqref{ARGMIN} is well-posed and~Theorem~\ref{TheoremHyp}~ensures hyperbolicity
for all basis functions. 
For solutions that are closer to the zero-level set, the optimization problem~\eqref{ARGMIN} may have a solution that results in non-positive realizations~$
	\GK \big[\widehat{ \,
	\lVert \u \rVert
	\, }\big]\big(\xi(\omega)\big) \leq0$. Then, the matrix~\eqref{AI} may be indefinite~\cite[Th.~2.1]{H3}. 
These, negative realizations may result from projection errors, when the Galerkin method is used for higher polynomial ansatz functions. 
Since the presented intrusive approach allows to precompute all stochastic quantities exactly,  
an efficient criterium~\eqref{QuadratureNodes} that verifies the assumptions of Theorem~\ref{TheoremHyp} is to test positivity of the truncated polynomial chaos expansions at the quadrature nodes.

The case, when states are close to the boundary of the random interface, where the minimization problem~\eqref{ARGMIN} is ill-posed, involves additional theoretical and numerical challenges. 
It is an active field of research to introduce regularization techniques~\cite{Alldredge2012} and piecewise ansatz function, as e.g.~wavelet basis, that exploit the local structures.


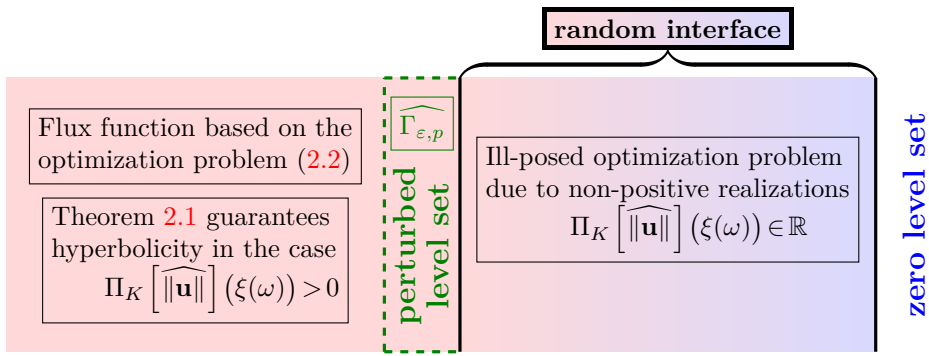
\begin{figure}[H]
	\begin{center}

\begin{tikzpicture}[scale=1]
	\tikzstyle{important line}=[very thick]

	\begin{scope}
		
		\shade [left color=red, right color=blue, opacity=0.15] (0,-2) rectangle (5.5,1.6);

		\shade [left color=red, right color=blue, opacity=0.15] (1.1,2) rectangle (4.4,2.5);

	\end{scope}
	
	
	

	
	
	\begin{scope}
		
		\shade [left color=red, right color=red, fill opacity=0.15] (-6,-2) rectangle (0,1.6);
		
	\end{scope}

	
	\draw[style=important line,dashed,gruen] (-1,1.6) -- (0,1.6);
	\draw[style=important line,dashed,gruen] (-1,-2) -- (0,-2);
	\draw[style=important line,dashed,gruen] (-1,-2) -- (-1,1.6);
	\draw[style=important line] (0,-2) -- (0,1.6);
	\draw[style=important line] (5.5,-2) -- (5.5,1.6);

	
	
	\draw[decoration={calligraphic brace,amplitude=10pt}, decorate, line width=1.6pt] (0,1.6) node {} -- (5.5,1.6);

	
	

	

	
	
	\node[rotate=0] at (-0.5,1) {\color{gruen} $\boxed{\widehat{\Gamma_{\varepsilon,p}}}$};
	
	\node[rotate=90] at (-0.7,-0.6) {\color{gruen} \large\textbf{perturbed}};
	
	\node[rotate=90] at (-0.3,-0.6) {\color{gruen} \large\textbf{level set}};
	
	\node[rotate=90] at (6,-0.2) {\color{blue} \large\textbf{zero level set}};
	
	
	
	\node[draw, align=left] at (-3.5,0.7) {Flux function based on the\\ optimization problem~\eqref{ARGMIN}};
	
	\node[draw, align=left] at (-3.5,-0.8) {
		Theorem~\ref{TheoremHyp} guarantees \\ hyperbolicity
		in the case\\ \qquad$ 
		\GK \left[\widehat{ \,
			\lVert \u \rVert
			\, }\right]
		\big(\xi(\omega)\big) >0$
	};
	
	
	
	
	\setlength\fboxrule{1.6pt}
	\draw (2.75,2.25) -- (2.75,2.25) node {\boxed{\textbf{random interface}}};
	
	\node[draw, align=left] at (2.75,0) {
		Ill-posed optimization problem\\ due to  non-positive realizations \\
		\hspace*{1cm} $ 
		\GK \left[\widehat{ \,
			\lVert \u \rVert
			\, }\right]
		\big(\xi(\omega)\big)  \in\mathbb{R}$
	};
	

\end{tikzpicture}
	\end{center}
	\caption{Illustration of the theoretical results: The left, red area includes states that satisfy the assumptions of~Theorem~\ref{TheoremHyp} for all basis functions. A perturbed level set  \eqref{randomLS} is green, dashed exemplified. 
		The right, blue domain illustrates the zero-level set. The random interface is a region, where expansions  may result in both positive and vanishing realizations such that assumption \eqref{AI} is violated.}
	\label{FigureAdaptivity}
\end{figure}


\section*{Appendix}
For the sake of completeness we state an example when the presented conservative form loses its hyperbolic character. Furthermore, the numerical discretizations in Section~\ref{SectionFV} are extended to space-varying velocities and two dimensions. 

{
\appendix

\section{Loss of hyperbolicity} \label{SecCounterExample}
To give an example such that the Jacobian~$\JoneDim \big(\widehat{\u},\v\big)$ 
of 
the conservative form yields complex eigenvalues, we consider the states~$
\widehat{\u}=(5,2,-1)^\T
$, $
{\v}=(0,20,2)^\T
$ and use normalized Hermite polynomials. 
More precisely, \cite[example~12.1]{S18} states that  the matrices~\eqref{2ndMoment} read~as
$$
\mathcal{M}_0 = \begin{pmatrix} 1 \\ & 1 \\ & & 1 \end{pmatrix}, 
\quad
\mathcal{M}_1 = \begin{pmatrix}  & 1 \\ 1 & & \sqrt{2} \\ & \sqrt{2}&  \end{pmatrix},
\quad
\mathcal{M}_2 = \begin{pmatrix}  & & 1 \\  & \sqrt{2} \\ 1 & & \sqrt{8} \end{pmatrix}
$$
when these polynomials and the truncation~$K=2$ are used. 
Hence, the stochastic Galerkin matrix and the Galerkin product are given by

$$
\P(\hat{\alpha}) 
=
\begin{pmatrix}
	\hat{\alpha}_0 & \hat{\alpha}_1 & \hat{\alpha}_2 \\
	\hat{\alpha}_1 & \hat{\alpha}_0+\sqrt{2}\hat{\alpha}_2 & \sqrt{2} \hat{\alpha}_1 \\
	\hat{\alpha}_2 & \sqrt{2} \hat{\alpha}_1 & \hat{\alpha}_0+\sqrt{8} \hat{\alpha}_2
\end{pmatrix}
\quad\text{and}\quad
\P(\hat{\alpha}) \hat{\alpha}
=
\begin{pmatrix}
	\lVert \hat{\alpha} \rVert^2 \\
	2\hat{\alpha}_1 (\hat{\alpha}_0 + \sqrt{2} \hat{\alpha}_2) \\
	\sqrt{2}  \hat{\alpha}_1^2 + \sqrt{8} \hat{\alpha}_2^2 + 2 \hat{\alpha}_0 \hat{\alpha}_2
\end{pmatrix}.
$$
Then, the nonlinear system~$\P(\hat{\alpha}) \hat{\alpha}
=
\P(\widehat{\u}) \widehat{\u}
$ 
has the four solutions
$$
\hat{\alpha}^\pm 
\approx 
\pm
\begin{pmatrix}
	\hphantom{-} 5.168558220676993 \\
	\hphantom{-} 1.690363139290745 \\
	-0.654735348671055
\end{pmatrix}
\quad\text{and}\quad
\tilde{\alpha}^\pm
=
\pm
\begin{pmatrix}
	\hphantom{-}5 \\ \hphantom{-}2 \\ -1
\end{pmatrix}.
$$
The matrices~$\P\big( \tilde{\alpha}^\pm \big)$ are indefinite, the matrix~$\P\big( \hat{\alpha}^- \big)$ is negative definite and $\P\big( \hat{\alpha}^+ \big)$ is strictly positive definite. Therefore, the unique solution to the optimization problem~\eqref{ARGMIN} is identified by~$
\hat{\alpha}^+
$ and the Jacobian to the conservative form reads as
$$
\JoneDim \big(\widehat{\u},\v\big)
=
\P\big(\v\big) 
\P\big( \widehat{ \,
	\lVert \u \rVert
	\, } \big)^{-1}
\P\big(\widehat{\u}\big)
\quad\text{for}\quad
\widehat{ \,
	\lVert \u \rVert
	\, }
=
\hat{\alpha}^+.
$$
Its spectrum~$\sigma\left\{
\JoneDim \big(\widehat{\u},\v\big)
\right\}
\approx
\big\{
0.01, \;
30.73 \pm 9.97 i
\big\}
$ is complex. 
In comparison, the eigenvalue estimate to the hyperbolicity preserving scheme is
$$
\vMAX \
\max 
\Big\{ 
\sigma\left\{ 
\JtildeoneDim
\big( \widehat{\u} \big) \right\}
\Big\}
\approx
38.97
\quad\text{with}\quad
\sigma\left\{ 
\JtildeoneDim
\big( \widehat{\u} \big) \right\}
\approx
\big\{
0.93, \; \pm 1
\big\}. 
$$


\section{Two-dimensional numerical discretizations with space-dependent velocity}
We consider a 
two-dimensional  
domain with  equidistant discretization~$
\DxI=x_{i,j+\nicefrac{1}{2},j}- x_{i,j-\nicefrac{1}{2}} 
$,  $\DxII=x_{i+\nicefrac{1}{2},j}- x_{i-\nicefrac{1}{2},j} $ and 
grid cells~$C_{i,j}\subset \mathbb{R}^2 $ 
centered at~$\x_{i,j}$. The cell averages   are denoted by~$\ubar_{i,j}^n$ and~$\wbar_{i,j}^n$. 

\subsubsection*{Conservative form} 
The  discretization~\eqref{NumFluxConserv}~--~\eqref{CFLconservative}, which is introduced in~Section~\ref{SecDiscretizationConservative}, reads in the multidimensional case as 
\begin{align*}
	\frac{
		\ubar_{i,j}^{n+1}
		-
		\ubar_{i,j}^{n}
	}{\Delta t}
	&=
	-
	\frac{
		\NumFlux_1\big(\ubar_{i,j}^{n},\ubar_{i,j+1}^{n},\v( \x_{i,j+\nicefrac{1}{2}} )\big)
		-
		\NumFlux_1\big(\ubar_{i,j-1}^{n},\ubar_{i,j}^{n},\v(\x_{i,j-\nicefrac{1}{2}})\big)
	}{\DxI} \\
	& 
	\hphantom{=}
	-
	\frac{
		\NumFlux_2\big(\ubar_{i,j}^{n},\ubar_{i+1,j}^{n},\v(\x_{i+\nicefrac{1}{2},j})\big)
		-
		\NumFlux_2\big(\ubar_{i-1,j}^{n},\ubar_{i,j}^{n},\v(\x_{i-\nicefrac{1}{2},j})\big)
	}{\DxII}, \\
	\NumFlux_i\big(\ubarL,\ubarR,\v\big) 
	&\coloneqq 
	\frac{1}{2} \Big( \widehat{f_i}\big(\ubarL,\v\big)+\widehat{f_i}\big(\ubarR,\v\big) \Big) 
	- 
	\frac{1}{2}
	\max\limits_{q=\ell,r} \Big\{ 
	\sigma\Big\{ 
	\Ji
	\big( \ubarK,\v\big) \Big\}
	\Big\}
	\big( \ubarR - \ubarL \big).
\end{align*}
Here, the 
CFL-condition  
\begin{equation*}
	\max_{i,j}\bigg\{
	\sigma\left\{ 
	\J \big(\ubar_{i,j}^n,\v\big(\x_{i\pm\nicefrac{1}{2},j\pm\nicefrac{1}{2}})\big)
	\right\}
	\bigg\}
	\frac{\Delta t}{\Delta \x}
	<1
\end{equation*}
is determined by the Jacobian to the conservative form,  which is given in Theorem~\ref{TheoremHyp}, as long as its spectrum is  real.


\subsubsection*{Capacity form} 
The hyperbolicity preserving  discretization, which is introduced  in~Section~\ref{SectionNumericsCapacity}, reads in the two-dimensional case as 
\begin{align*}
	\frac{
		\wbar_{i,j}^{n+1}
		-
		\wbar_{i,j}^{n}
	}{\Delta t}
	&=
	-
	\frac{
		\NumFluxX_1\big(\wbar_{i,j}^{n},\wbar_{i,j+1}^{n},\v( \x_{i,j} )\big)
		-
		\NumFluxX_1\big(\wbar_{i,j-1}^{n},\wbar_{i,j}^{n},\v(\x_{i,j})\big)
	}{\DxI} \\
	& 
	\hphantom{=}
	-
	\frac{
		\NumFluxX_2\big(\wbar_{i,j}^{n},\wbar_{i+1,j}^{n},\v(\x_{i,j})\big)
		-
		\NumFluxX_2\big(\wbar_{i-1,j}^{n},\wbar_{i,j}^{n},\v(\x_{i,j})\big)
	}{\DxII} \\
	&\hphantom{=} -	\big(\partial_{x_1}+\partial_{x_2} \big)\v(\x)\big|_{\x=\x_{i,j}} \;
	\ast
	\widehat{ \
		\lVert 	\wbar_{i,j}^{n} \rVert
		\ }, \\
	\NumFluxX_i\big(\wbarL,\wbarR,\v\big) 
	&\coloneqq 
	\frac{1}{2} \Big( \widetilde{F_i}\big(\wbarL,\v\big)+\widetilde{F_i}\big(\wbarR,\v\big) \Big) 
	- 
	\frac{1}{2}
	\max\limits_{q=\ell,r} \Big\{ 
	\sigma\Big\{ 
	\JC
	\big( \wbarK \big) \Big\}
	\Big\}
	\big( \wbarR - \ubarL \big), \\
	\widetilde{F}_i\big(\wbar,\v \big)
	&\coloneqq
	\Big[
	\diag\{1,1\}\otimes 
	\DP \big(\v(x)\big) 
	\VP(\v)^\T 
	\Big]
	\,
	\widetilde{f_i}\big( \VP(\v) \wbar \big).
\end{align*}

\noindent
The CFL-condition for the capacity form~is based on the real spectrum~of the Jacobian~$\JC (\ubar)
=
\JC
\big( \VP(\v) 	\wbar \big) 
$ in Theorem~\ref{TheoremHyp} and the velocity is evaluated at the cell center, which yields the estimate
$$
\vMAX \
\max\limits_{i,j} \Big\{ 
\sigma\left\{ 
\JC
\big( 	\ubar_{i,j}^{n} \big) \right\}
\Big\}
\frac{\Delta t}{\Delta \x}
<1
\quad\text{with}\quad
\vMAX
\coloneqq
\max\limits_{\stackrel{k=0,\ldots,K,}{i,j}} \Big\{ \Big|
\DP_k \big(\v(\x_{i,j} )\big) \Big|
\Big\}.
$$

}

\bigskip

\section*{Acknowledgments}
The authors thank the Deutsche Forschungsgemeinschaft (DFG,
German Research Foundation) for the financial support through
projects BA4253/11-2 and HE5386/19-2 of the priority program
2183 ``Property-Controlled Forming Processes''. 
Furthermore, this work is supported by the PRIME programme of the German Academic Exchange Service (DAAD).

\bibliographystyle{plain}
\bibliography{BGKbib}

\begin{thebibliography}{10}

\bibitem{Alldredge2012}
G.~Alldredge, C.~Hauck, and A.~Tits.
\newblock High-order entropy-based closures for linear transport in slab
  geometry~{II}: {A} computational study of the optimization problem.
\newblock {\em SIAM Journal on Scientific Computing}, 34(4):361--391, 2012.

\bibitem{Imran2021}
M.~Bambach, M.~Imran, I.~Sizova, J.~Buhl, S.~Gerster, and M.~Herty.
\newblock A soft sensor for property control in multi-stage hot forming based
  on a level set formulation of grain size evolution and machine learning.
\newblock {\em Advances in Industrial and Manufacturing Engineering}, 2:100041,
  2021.

\bibitem{Caselles1992}
V.~Caselles.
\newblock Scalar conservation laws and {H}amilton-{J}acobi equations in
  one-space variable.
\newblock {\em Nonlinear Analysis: Theory, Methods \& Applications},
  18(5):461--469, 1992.

\bibitem{Clarke1990}
F.~H. Clarke.
\newblock {\em Optimization and Nonsmooth Analysis}.
\newblock Society for Industrial and Applied Mathematics, 1990.

\bibitem{LipschitzFlux2001}
J.~Correia, P.~LeFloch, and M.~Thanh.
\newblock Hyperbolic systems of conservation laws with {L}ipschitz continuous
  flux-functions: {T}he {R}iemann problem.
\newblock {\em Boletim da Sociedade Brasileira de Matemática}, 32:271--301,
  2001.

\bibitem{Crandall1983}
M.~G. Crandall and P.-L. Lions.
\newblock Viscosity solutions of {H}amilton-{J}acobi equations.
\newblock {\em Transactions of the American Mathematical Society}, 277:1--42,
  1983.

\bibitem{Epshteyn2021}
D.~Dai, Y.~Epshteyn, and A.~Narayan.
\newblock Hyperbolicity-preserving and well-balanced stochastic {G}alerkin
  method for shallow water equations.
\newblock {\em SIAM Journal on Scientific Computing}, 43:A929--A952, 01 2021.

\bibitem{Epshteyn2022}
D.~Dai, Y.~Epshteyn, and A.~Narayan.
\newblock Hyperbolicity-preserving and well-balanced stochastic {G}alerkin
  method for two-dimensional shallow water equations.
\newblock {\em Journal of Computational Physics}, 452:110901, 2022.

\bibitem{S15}
B.~J. Debusschere, H.~N Najm, P.~P. P\'ebay, O.~M. Knio, R.~G. Ghanem, and
  O.~P.~Le Ma\^itre.
\newblock Numerical challenges in the use of polynomial chaos representations
  for stochastic processes.
\newblock {\em {SIAM} Journal on Scientific Computing}, 26(2):698--719, 2004.

\bibitem{Dervieux1980}
A.~Dervieux and F.~Thomasset.
\newblock A finite element method for the simulation of a {R}ayleigh-{T}aylor
  instability.
\newblock In Reimund Rautmann, editor, {\em Approximation Methods for
  Navier-Stokes Problems}, pages 145--158, Berlin, Heidelberg, 1980. Springer
  Berlin Heidelberg.

\bibitem{H0}
B.~Despr\'es, G.~Po\"ette, and D.~Lucor.
\newblock Uncertainty quantification for systems of conservation laws.
\newblock {\em Journal of Computational Physics}, 228:2443--2467, 2009.

\bibitem{H1}
B.~Despr\'es, G.~Po\"ette, and D.~Lucor.
\newblock {\em Robust uncertainty propagation in systems of conservation laws
  with the entropy closure method}, volume~92 of {\em Uncertainty
  Quantification in Computational Fluid Dynamics. Lecture Notes in
  Computational Science and Engineering}.
\newblock Springer, Cham, 2013.

\bibitem{Frankowska1989}
H.~Frankowska.
\newblock Hamilton-{J}acobi equations: {V}iscosity solutions and generalized
  gradients.
\newblock {\em Journal of Mathematical Analysis and Applications}, 141:21--26,
  1989.

\bibitem{GersterHertyCicip2020}
S.~Gerster and M.~Herty.
\newblock Entropies and symmetrization of hyperbolic stochastic {G}alerkin
  formulations.
\newblock {\em Communications in Computational Physics}, 27:639--671, 2020.

\bibitem{GersterJCP2019}
S.~Gerster, M.~Herty, and A.~Sikstel.
\newblock Hyperbolic stochastic {G}alerkin formulation for the $p$-system.
\newblock {\em Journal of Computational Physics}, 395:186--204, 2019.

\bibitem{Godlewski1998}
E.~Godlewski and P.~A. Raviart.
\newblock {\em Numerical Approximation of Hyperbolic Systems of Conservation
  Laws}.
\newblock Applied Mathematical Sciences, Springer, New York, 1996.

\bibitem{H3}
D.~Gottlieb and D.~Xiu.
\newblock {G}alerkin method for wave equations with uncertain coefficients.
\newblock {\em Communications in computational physics}, 3(2):505--518, 2008.

\bibitem{HamiltonJacobiJin}
J.~Hu, S.~Jin, and D.~Xiu.
\newblock A stochastic {G}alerkin method for {H}amilton-{J}acobi equations with
  uncertainty.
\newblock {\em SIAM Journal on Scientific Computing}, 37(5):A2246--A2269, 2015.

\bibitem{Jin1998}
S.~Jin and Z.~Xin.
\newblock Numerical passage from systems of conservation laws to
  {H}amilton-{J}acobi equations, and relaxation schemes.
\newblock {\em SIAM Journal on Numerical Analysis}, 35(6):2385--2404, 1998.

\bibitem{H5}
S.~Jin, D.~Xiu, and X.~Zhu.
\newblock A well-balanced stochastic {G}alerkin method for scalar hyperbolic
  balance laws with random inputs.
\newblock {\em Journal of Scientific Computing}, 67:1198--1218, 2016.

\bibitem{Kruskov1975}
S.~N. Kruzkov.
\newblock Generalized solutions of the {H}amilton-{J}acobi equations of eikonal
  type.
\newblock {\em Mathematics of the {USSR}-Sbornik}, 27(3):406--446, 1975.

\bibitem{kusch2019maximum}
J.~Kusch, G.~W. Alldredge, and M.~Frank.
\newblock Maximum-principle-satisfying second-order intrusive polynomial moment
  scheme.
\newblock {\em The SMAI journal of computational mathematics}, 5:23--51, 2019.

\bibitem{Leveque}
R.~J. Leveque.
\newblock {\em Finite volume methods for hyperbolic problems}.
\newblock Cambridge Texts in Applied Mathematics. Cambridge University Press, 1
  edition, 2002.

\bibitem{S4}
O.~P.~Le Ma\^itre and O.~M. Knio.
\newblock {\em Spectral Methods for uncertainty quantification}.
\newblock Springer Netherlands, 1 edition, 2010.

\bibitem{Osher1988}
S.~Osher and J.~A. Sethian.
\newblock Fronts propagating with curvature-dependent speed: {A}lgorithms based
  on {H}amilton-{J}acobi formulations.
\newblock {\em Journal of Computational Physics}, 79(1):12--49, 1988.

\bibitem{Preusser}
T.~P\"{a}tz and T.~Preusser.
\newblock Segmentation of stochastic images using level set propagation with
  uncertain speed.
\newblock {\em Journal of Mathematical Imaging and Vision volume},
  48(3):467–487, 2014.

\bibitem{Pettersson2019}
P.~Pettersson, A.~Doostan, and J.~Nordstr\"om.
\newblock Level set methods for stochastic discontinuity detection in nonlinear
  problems.
\newblock {\em Journal of Computational Physics}, 392:511--531, 2019.

\bibitem{S5}
P.~Pettersson, G.~Iaccarino, and J.~Nordstr\"om.
\newblock A stochastic {G}alerkin method for the {E}uler equations with {R}oe
  variable transformation.
\newblock {\em Journal of Computational Physics}, 257:481--500, 2014.

\bibitem{Roe1}
P.~L. Roe.
\newblock Approximate {R}iemann solvers, parameter vectors, and difference
  schemes.
\newblock {\em Journal of Computational Physics}, 43:357--372, 1981.

\bibitem{Sethian2003}
J.~Sethian and P.~Smereka.
\newblock Level set methods for fluid interfaces.
\newblock {\em Annual Review of Fluid Mechanics}, 35:341--372, 2003.

\bibitem{Sethian1999}
J.~A. Sethian.
\newblock {\em Level Set Methods and Fast Marching Methods: {E}volving
  Interfaces in Computational Geometry}.
\newblock 1999.

\bibitem{Sonday2011}
B.~Sonday, R.~Berry, H.~Najm, and B.~Debusschere.
\newblock Eigenvalues of the {J}acobian of a {G}alerkin-projected uncertain
  {ODE} system.
\newblock {\em Journal of Scientific Computing}, 33:1212--1233, 2011.

\bibitem{Nouy2009}
G.~Stefanou, A.~Nouy, and A.~Clement.
\newblock Identification of random shapes from images through polynomial chaos
  expansion of random level set functions.
\newblock {\em International Journal for Numerical Methods in Engineering},
  79(2):127--155, 2009.

\bibitem{Subbotina2017}
N.~Subbotina and L.~Shagalova.
\newblock Generalized solutions of {H}amilton–{J}acobi equation to a
  molecular genetic model.
\newblock pages 462--471, 04 2017.

\bibitem{S18}
T.~J. Sullivan.
\newblock {\em Introduction to uncertainty quantification}.
\newblock Texts in Applied Mathematics. Springer, Switzerland, 1 edition, 2015.

\bibitem{Beckermann2007}
Y.~Sun and C.~Beckermann.
\newblock Sharp interface tracking using the phase-field equation.
\newblock {\em Journal of Computational Physics}, 220(2):626--653, 2007.

\end{thebibliography}

\end{document}